\documentclass[10pt, twocolumn]{article}
\usepackage{amssymb,amsfonts,amsmath,amsthm}
\usepackage{adjustbox}
\usepackage{graphicx}
\usepackage{enumerate}
\usepackage[round]{natbib}
\usepackage{url} % not crucial - just used below for the URL 
\usepackage{tikz}
\usepackage{centernot}
\usepackage{setspace}

\newcommand{\blind}{1}

\newtheorem{lem}{Lemma}
\newtheorem{cor}{Corollary}
\newtheorem{prp}{Proposition}
\newtheorem{rmk}{Remark}
\newtheorem{thm}{Theorem}
\newtheorem{exm}{Example}
\newtheorem{dfn}{Definition}

\DeclareMathOperator*{\esssup}{ess\,sup}
\DeclareMathOperator*{\essinf}{ess\,inf}

\begin{document}

	\def\spacingset#1{\renewcommand{\baselinestretch}%
	{#1}\small\normalsize} \spacingset{1}

	%%%%%%%%%%%%%%%%%%%%%%%%%%%%%%%%%%%%%%%%%%%%%%%%%%%%%%%%%%%%%%%%%%%%%%%%%%%%%%
	
	\if1\blind
	{
	  \title{\bf Post-hoc $\alpha$ Hypothesis Testing and the Post-hoc $p$-value}
	  \author{Nick W. Koning\hspace{.2cm}\\
	    Econometric Institute, Erasmus University Rotterdam, the Netherlands}
	  \maketitle
	} \fi
	
	\if0\blind
	{
	  \bigskip
	  \bigskip
	  \bigskip
	  \begin{center}
	    {\LARGE\bf Post-hoc $\alpha$ Hypothesis Testing}
	\end{center}
	  \medskip
	} \fi
	
	\bigskip
	\begin{abstract}
		In traditional hypothesis testing one must pre-specify the significance level $\alpha$ to bound the `size' of the test: its probability to falsely reject the hypothesis.
		Indeed, a data-dependent selection of $\alpha$ would generally distort the size, possibly making it larger than the specified level $\alpha$.
		We explore hypothesis testing with a data-dependent choice of $\alpha$ by guaranteeing that there is no such size distortion in expectation, even if the level $\alpha$ is arbitrarily selected based on the data.
		Unlike regular $p$-values, resulting `post-hoc $p$-values' allow us to `reject at level $p$' and still provide this guarantee.
		Interestingly, we find that $p$ is a post-hoc $p$-value if and only if $1/p$ is an $e$-value, a recently introduced measure of evidence.
		While often treated as different paradigms, this reveals $e$-values are simply $p$-values under a stronger error guarantee, thinly veiled by the reciprocal $p = 1/e$.
		Moreover, we extend classical optimal testing to optimal post-hoc testing.
		Finally, we apply our work to close Markov's inequality into a post-hoc $\alpha$ equality, and we study more general forms of post-hoc testing that require us to generalize beyond $e$-values.
	\end{abstract} \\
	
	\noindent%
	{\it Keywords:} $p$-values, $e$-values, $p$-hacking, data-dependent level.
	%\vfill
	%\tableofcontents

\section{Introduction}
		Testing hypotheses is the cornerstone of the modern scientific method.
		An unfortunate feature of traditional hypothesis testing is that one must pre-specify a significance level `$\alpha$' to bound the \emph{size} or \emph{Type I error}: the probability to falsely reject a true hypothesis.
		This pre-specification of $\alpha$ has shaped empirical scientific discourse over the past century, and has led to several widely known problems:
		
		\begin{itemize}
			\item \textbf{$\alpha$-hacking}. The process of specifying the level $\alpha$ is typically not publicly observed, so that outsiders cannot verify whether it was truly pre-specified or secretly selected post-hoc: after seeing the data. To protect against accusations of such `$\alpha$-hacking', it is therefore near-universal practice to use a standardized level such as $\alpha = .05$.
			\item \textbf{File-drawer and publication bias}. 
			A finding with a $p$-value larger than the pre-specified $\alpha$ cannot be claimed as a discovery.
			This makes such findings substantially less likely to be published or even pursued for publication.
			This is known as the file-drawer problem, as these findings are often indefinitely relegated to a file-drawer \citep{rosenthal1979file}.
			As a consequence, the scientific literature often contains a biased selection of all the collected evidence.
			
			\item \textbf{$p$-hacking}. The use of a pre-specified $\alpha$ incentivizes `$p$-hacking': the practice of modifying the analysis based on the data in order to push a $p$-value below the pre-specified $\alpha$ \citep{simonsohn2014p}.
		\end{itemize}
		%These issues, combined with pervasive misinterpretations of $p$-values lead to a statement on $p$-values by the American Statistical Association, and 21 commentaries \citep{wasserstein2016statement}.
		%The idea is to stop viewing the $p < \alpha$ as a continuous 
		%Special issue in the American Statistician titled ``Moving to a World Beyond `$p < .05$'", containing 43 articles on what next \citep{wasserstein2019moving}
		
		These problems and their downstream effects, combined with pervasive misinterpretation, have led to the radical proposal to abandon traditional hypothesis testing altogether \citep{amrhein2018remove, amrhein2019scientists, mcshame2019abandon, wasserstein2019moving}.
		Unfortunately, doing so would simultaneously discard a guarantee on the probability of making a false discovery.
		%As a replacement, a returning suggestion is to interpret the $p$-value as a continuous measure of evidence, rather than using a dichotomous `significance' decision $p \leq \alpha$.
		%Unfortunately, this simultaneously discards the guarantee that a hypothesis is falsely rejected with probability at most $\alpha$.
		
		%Rather than abandoning traditional hypothesis testing, we aim to generalize it by exploring the use of data-dependent levels.
		
		We study a fundamental solution to these problems: testing under a data-dependent selection of the level $\alpha$.
		Selecting the level $\alpha$ based on the data causes a distortion of the size compared to the selected level.
		We quantify this size distortion using the ratio $\text{size} / \alpha$ and develop a theory of \emph{post-hoc $\alpha$ hypothesis testing} under the guarantee that this size distortion is at most 1 in expectation for \emph{any} data-dependent choice of the level $\alpha$.
		Traditional hypothesis tests only offer this guarantee for pre-specified (or independently specified) levels $\alpha$.
		
%		Rather than discarding this guarantee, we propose to replace it with a stronger `post-hoc' guarantee that allows us to select the level $\alpha$ `post-hoc': based on the data.
%		Selecting $\alpha$ based on the data causes a size distortion compared to the selected level.
%		We measure the size distortion as $\text{size} / \alpha$, and develop a theory of testing with the guarantee that this distortion is at most 1 in expectation for all data-dependent levels $\alpha$.
		
		%A natural choice for the data-dependent level is the smallest data-dependent level for which we reject: the $p$-value.
		If we intend to make the most powerful claim possible, then we should use the smallest data-dependent level $\alpha$ for which we reject: the $p$-value.
%		A consequence is that we do not have to think about how we should pre-specify the level $\alpha$.
%		Indeed, we can simply set it equal to the smallest data-dependent level for which we reject: the $p$-value.
		We call the $p$-value of a post-hoc $\alpha$ hypothesis test a \emph{post-hoc $p$-value}.
		Such post-hoc $p$-values are also $p$-values in the traditional sense, but their corresponding rejection decisions offer a much stronger error guarantee.
		
		Indeed, we can truly `reject at level $p$' when using a post-hoc $p$-value $p$, and have the guarantee that there is no size distortion in expectation.
		The claim associated with a traditional $p$-value is much weaker: we can only interpret it as `the smallest level at which we would have rejected, had we pre-specified $\alpha = p$', if we want to maintain its guarantee.
			
		A major side-benefit is that post-hoc $p$-values from multiple studies are easily combined.
		In particular, multiplying independent post-hoc $p$-values together yields another post-hoc $p$-value.
		Moreover, a harmonic average of post-hoc $p$-values is still a post-hoc $p$-value.
		Traditional $p$-values are much less straightforward to combine.
		
		Post-hoc $p$-values and tests are surprisingly easy to derive.
		In particular, we find that $p$ is a post-hoc $p$-value if and only if the expectation of its reciprocal is at most one: $\mathbb{E}\, 1/p \leq 1$ under our hypothesis.
		This means that $p$ is a post-hoc $p$-value when its reciprocal is an \emph{$e$-value}: a recently proposed measure of evidence \citep{howard2021time, shafer2021testing, vovk2021values, grunwald2023safe, ramdas2023gametheoretic}.
		Many such $e$-values have already been developed, constituting a rich source of post-hoc $p$-values.
		
		%In fact, this relationship reveals the true nature of $e$-values.
		While $e$-values and $p$-values are often presented as entirely different paradigms, the above relationship reveals a deep connection.
		Indeed, it shows that $e$-values \emph{are} $p$-values under a stronger error guarantee, thinly veiled by a reciprocal operation: $e = 1/p$.
		This offers a very precise answer to the often-posed question of what $e$-values have to offer compared to traditional $p$-values.
		Moreover, this provides a decision-theoretic foundation for the $e$-value in the context of a binary hypothesis testing problem.
		In addition, as $p$-values are familiar to all statisticians, we believe this connection may make $e$-values more palatable to a wider audience, who may be hesitant to adopt a new paradigm without understanding what is gained compared to an existing paradigm.
		
		A downside of post-hoc $p$-values is that they are generally larger than traditional $p$-values.
		This is a direct consequence of the fact that the post-hoc error guarantee is stronger than the traditional error guarantee.
		
		Interestingly, there does not seem to be a unique notion of post-hoc power, as the traditional notion of power does not generalize.
		Indeed, we would naively like to maximize the rejection probability under the alternative at the selected data-dependent level.
		Moreover, we would like to use the smallest data-dependent level at which we reject: the post-hoc $p$-value.
		However, since we always reject at this level, there is no rejection probability left to maximize.
		
		The natural replacement of power is to make the post-hoc $p$-value `small' under the alternative.
		This comes with a lot of flexibility: we can choose whether we want a moderately small $p$-value with high probability or gamble for a tiny $p$-value with very small probability, depending on the application.
		We tackle this by using the ``Neyman-Pearson lemma for $e$-values'' recently derived in \citet{koning2024continuous}.
		We also recover the classical Neyman-Pearson lemma as a special case, replacing randomization by a rejection at a larger post-hoc level.
		
		We also include two technical sections.
		In Section \ref{sec:markov}, we show how our theory can be used to tighten Markov's inequality to a ``Markov's equality"; in fact, this was the original inspiration of this work.
		In Section \ref{sec:beyond_e-values}, we discuss the abstract theory underlying post-hoc testing.
	
		\subsection{Contributions to the literature}
			Preceding work on the post-hoc selection of a level may be found in \citet{katsevich2020simultaneous, xu2022post, wang2022false, grunwald2022beyond}.
			In particular, the fact that $e$-values yield a type of post-hoc valid decisions was also recently observed in multiple testing \citep{katsevich2020simultaneous, xu2022post, wang2022false}. 
			However, this work notably does not cover the necessity of $e$-values for this kind of post-hoc $\alpha$ testing, nor does it explore its use when testing a single hypothesis: it only considers error bounds in multiple testing.
			The closest precursor is the work of \citet{grunwald2022beyond}, who does connect $e$-values to the problem of testing with a data-dependent $\alpha$.
			However, he considers a very different perspective on post-hoc decisions, notably missing the connection to $p$-values as well as the necessity of $e$-values for post-hoc testing: we show every (non-dominated) post-hoc test is of the form $e \mapsto \mathbb{I}\{e \geq 1/\alpha\}$.
			In Appendix \ref{sec:grunwald} we discuss the connection to this work.
			
			On top of these differences, we bring five broad contributions to this literature.
			First, we derive a novel integrative framework of post-hoc testing that consolidates prior theory.
			Here, we build testing with data-dependent levels from the ground-up, starting from traditional hypothesis testing.
			Perhaps the key observation here is that one must \emph{choose} how to handle size distortions, and we motivate and focus on the option that leads to the $e$-value.
			Second, while the preceding literature presents the $e$-value and $p$-value as two opposing paradigms, we show that $e$-values \emph{are} $p$-values derived under a stronger data-dependent-level Type I error guarantee, thinly veiled by the reciprocal $e = 1/p$.
			This concretely answers the open question in the $e$-value literature of what $e$-values offer compared to $p$-values.
			Third, we show how to capture optimality in the context of post-hoc level testing, using a utility-based framework that yields a risk-reward trade-off for rejections across different levels.
			Fourth, we nest post-hoc hypothesis testing in a more general theory of evidence, beyond $e$-values.
			Finally, we connect post-hoc $\alpha$ testing to Markov's inequality, revealing a ``Markov's equality''.
			As concentration inequalities generally rely on Markov's inequality, we show how this may be used to produce `concentration equalities', which we apply to Ville's-type inequalities.
			
			In the appendix, we further explore what lies beyond a single testing problem, exploring post-hoc $\alpha$ sequential testing and multiple testing.
			Based on earlier versions of this manuscript, \citet{hartog2025family} recently continued work on the post-hoc Familywise Error Rate, and \citet{xu2025bringing} applied these ideas to general multiple testing problems.
			In addition, we also introduce a notion of double post-hoc $p$-values which are $p$-values that are post-hoc valid under the hypothesis and have a reciprocal that is post-hoc valid under the alternative.
			An example of such a double post-hoc $p$-value is a likelihood ratio and its composite generalization (the `numeraire $e$-value') recently introduced by \citet{larsson2024numeraire}.
			Similar double interpretations have been attributed to likelihood ratios before, in the context of Bayes factors \citep{jeffreys1935some, jeffreys1998theory, kass1995bayes}.
			
			Our work has recently been applied in \citet{hemerik2024choosing} to provide a solution to the danger of having multiple `standard' significance levels in a single field of literature.
			Combining traditional $p$-values by taking their harmonic mean was recently suggested by \citet{wilson2019harmonic}.
			Our post-hoc $p$-values seem deeply related to this operation, as their harmonic mean is still a post-hoc $p$-value.
								
	\section{Traditional hypothesis testing}			
		\subsection{The problem}
			Any statistical problem starts with observing some data.
			The goal of statistical inference is to learn properties of the process that generated this data.
			Hypothesis testing is a specific statistical inference problem.
			There, we formulate a hypothesis about the data generating process and test whether it is true.
			% There, we formulate a hypothesis about the data generating process, and attempt to verify whether it is true.

			Unfortunately, we do not directly observe the data generating process: we only observe the data.
			This means that a test must somehow use the data to decide whether the hypothesis is true.
			We model such a test as a binary decision $\tau : \textnormal{data} \to \{0, 1\}$, where a $1$ means that our test rejects the hypothesis and $0$ that it does not reject the hypothesis.
						
			Making mistakes is inherent to hypothesis testing.
			To quantify such mistakes, it is customary to measure how often a test $\tau$ rejects our hypothesis when it is actually true.
			This probability of rejecting the hypothesis if it is true is also known as the \emph{size} of the test:
			\begin{align*}
				\textnormal{size}(\tau)
					:= \mathbb{P}(\tau = 1),
			\end{align*}
			sometimes also called the \emph{False Positive Rate} or \emph{Type I error}, where $\mathbb{P}$ denotes the probability under the hypothesis.\footnote{Our results and definitions extend to composite hypotheses, by replacing $\mathbb{P}$ with a supremum over the distributions that satisfy the hypothesis.}
			
			The standard strategy is to use a test with a very small size, so that if the hypothesis is true, then it would be implausible that the test rejects it.
			
		\subsection{Test functions}
			We typically want to use a test $\tau$ with a certain confidence guarantee expressed by a level $\alpha > 0$.
			In particular, say that a test is \emph{valid} at a level $\alpha$ if its probability of rejecting the hypothesis when it is true, the size, is at most $\alpha$:
			\begin{align*}
				\textnormal{size}(\tau)
					= \mathbb{P}(\tau = 1) \leq \alpha.
			\end{align*}
			It is common to call a test \emph{exactly} valid if this holds with equality.
			
			In practice, we often have access to an entire family or \emph{function} $\phi$ of tests, where each test $\phi(\alpha)$ is labeled by some level $\alpha > 0$.
			For example, it is common to have a test function $\phi$ of a threshold form $\phi(\alpha) = \mathbb{I}\{T > c(\alpha)\}$, which rejects the hypothesis when some \emph{test statistic} $T$ exceeds a \emph{critical value} $c(\alpha)$.
			By tuning the level $\alpha$, we can select a test of the desired level.
			
			We can extend the notion of a valid test to entire test functions, and call $\phi$ \emph{valid} if the size is below $\alpha$ for every $\alpha$:
			\begin{align}\label{eq:valid_original}
				\textnormal{size}(\phi(\alpha))  \leq \alpha,\quad \textnormal{ for all } \alpha.
			\end{align}
			This means that whichever $\alpha$ we pre-specify, the corresponding test $\phi(\alpha)$ is valid.
			
			It is typical to use a test function $\phi$ for which a smaller level $\alpha$ yields a more conservative test.
			For example, for a threshold-based test function $\mathbb{I}\{T > c(\alpha)\}$, the critical value typically increases as the level $\alpha$ decreases.
			More abstractly, we will assume throughout that our test functions are non-decreasing in $\alpha$, in the sense that if $\phi(\alpha^-)$ rejects then $\phi(\alpha^+)$ also rejects whenever $\alpha^+ \geq \alpha^-$.
			Moreover, we assume that given the data, there always exists some smallest $\alpha$ at which $\phi(\alpha)$ jumps from 0 to 1, eliminating test functions that equal 0 or 1 for every $\alpha$.
			A consequence of these assumptions is that $\phi$ is an indicator function in $\alpha$, as illustrated in Figure \ref{fig:test_family}.
			In the context of post-hoc $\alpha$ validity, the non-decreasing assumption is without loss of generality under the weak assumption that a user prefers rejecting for a smaller value of $\alpha$; see Remark \ref{rmk:non-decreasing}.
			
		\subsection{$p$-values}
			Under the specified conditions, a test function $\phi$ is completely summarized by the point at which the jump happens.
			This point is known as the $p$-value of $\phi$ \citep{lehmann2022testing}:
			\begin{align}\label{dfn:p-value}
				p = \min\{\alpha : \phi(\alpha) = 1\}.
			\end{align}
			All tests $\phi(\alpha)$ labeled by a level $\alpha < p$ do not reject, and all tests with a level $\alpha \geq p$ do reject.
			This means that for a $p$-value $p$, its corresponding test function $\phi$ can be written as $\alpha \mapsto \mathbb{I}\{p \leq \alpha\}$, where $\mathbb{I}\{p \leq \alpha\}$ denotes the indicator function of the event $p \leq \alpha$.
			
			A $p$-value is said to be \emph{valid} if its associated test function is.
			That is, its probability of being below $\alpha$ is bounded by $\alpha$:
			\begin{align}\label{eq:valid_p-value}
				\textnormal{size}(\mathbb{I}\{p \leq \alpha\})
					\equiv \mathbb{P}(p \leq \alpha) \leq \alpha,\quad \textnormal{ for all } \alpha.
			\end{align}

			\begin{figure}
				\centering
				\begin{tikzpicture}[scale=0.55]
				    % Draw axes
				    \draw [<->,thick] (0,5) node (yaxis) [above] {$\phi(\alpha)$}
				        |- (5,0) node (xaxis) [right] {$\alpha$};
				
				    % Draw cadlag function
				    \draw[black, thick] (0,0) -- (3,0); % continuous segment
				    \draw[black, thick] (3,4) -- (4.5,4); % continuous segment
				    
				    \draw[black] (3,0) -- (3,-0.2); % Tick mark
   					\node[below] at (3,-0.2) {$p$}; % Label for the tick mark
   					\draw[black, dashed] (3,0) -- (3,4); % continuous segment
   					
   					 % Draw jumps with dots
				    \fill[white] (3,0) circle (2pt);
				    \fill[black] (3,4) circle (2pt);
				    \draw[black] (3,0) circle (2pt);
				    
				    \draw[black] (0,4) -- (-0.2,4); % Tick mark
    				\node[below] at (-0.5, 4.5) {$1$}; % Label for the tick mark	
				\end{tikzpicture}
				\caption{Realization of a test function and its $p$-value.}
				\label{fig:test_family}
			\end{figure}
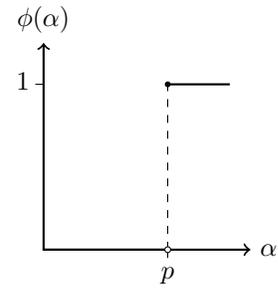
			
			\begin{rmk}
				Sometimes, \eqref{eq:valid_p-value} is used as the definition of a $p$-value.
				Our results still go through if this definition is taken as the starting point.
			\end{rmk}
			
		\subsection{Relative size distortion}
			If the size and the level $\alpha$ do not match, we speak of a \emph{size distortion}.
			At a given level $\alpha$, we can measure the size distortion by the ratio between the size and the level:
			\begin{align}\label{dfn:size_distortion}
				\frac{\textnormal{size}(\phi(\alpha))}{\alpha}
					\equiv \frac{\mathbb{P}(\phi(\alpha) = 1)}{\alpha}
					\equiv  \frac{\mathbb{P}(p \leq \alpha)}{\alpha}.
			\end{align}
			A test (function) or $p$-value is valid at $\alpha$ if its size distortion is at most 1.
			
			This relative measure of size distortion conveniently scales with the value of $\alpha$.
			For example, suppose that $\alpha = .01$ and $\textnormal{size} = .02$, then the size distortion equals $.02 / .01 = 2$.
			This is much larger than the desired value 1.
			At the same time, if  $\alpha = .10$ and $\textnormal{size} = .11$, the size distortion is much smaller: $.11 / .10 = 1.1$.
			
			\begin{rmk}\label{rmk:discrepancy}
				We stress that this way of measuring the size distortion is a choice.
				For example, we could alternatively measure the size distortion with $\textnormal{size}(\phi(\alpha))- \alpha$.
				However, this difference does not scale well with the value of $\alpha$.
				Indeed, the level $\alpha = .01$ and $\textnormal{size} = .02$ would then yield the same size distortion as $\alpha = .10$ and $\textnormal{size} = .11$.
				In practice, we believe the former is usually considered much more problematic, as expressed in using our relative notion of size distortion.
				For completeness, we also develop a notion of post-hoc $\alpha$ testing starting from this notion of size distortion in Example \ref{exm:harmonic}.
			\end{rmk}
			
	\section{Testing with data-dependent $\alpha$}
	
		\subsection{Generalizing size to data-dependent $\alpha$}
			Valid test functions and valid $p$-values are convenient, because they produce a valid test for each desired level $\alpha$.
			However, this ease of use also facilitates misuse: it becomes tempting to consider a data-dependent level $\widetilde{\alpha}$.
			For example, an analyst may secretly first look at the $p$-value, and then conveniently select the level based on this $p$-value.
			Unfortunately, such a procedure may have a size much larger than the selected level, as we illustrate in Example~\ref{exm:decreasing_alpha}.
			
			To study such a data-dependent selection of the level, we generalize the notion of size to data-dependent levels.
			For a data-dependent level, the natural definition of size is the conditional probability to falsely reject the hypothesis, given that the data-dependent level $\widetilde{\alpha}$ equals $a$:
			\begin{align*}
				\mathbb{P}(\phi(\widetilde{\alpha}) = 1 \mid \widetilde{\alpha} = a) \equiv \mathbb{P}(p \leq \widetilde{\alpha} \mid \widetilde{\alpha} = a).
			\end{align*}
			This can be interpreted as the actual rejection probability under the hypothesis if a rejection is claimed at level $a$, when using the data-dependent level $\widetilde{\alpha}$.
			
			For a pre-specified level $\widetilde{\alpha} = \alpha$, this simply reduces to the original definition of the size: $\mathbb{P}(\phi(\alpha) = 1 \mid \alpha = a) = \mathbb{P}(\phi(a) = 1) \equiv \mathbb{P}(p \leq a)$.
			
			\begin{exm}[$\alpha$-hacking]\label{exm:decreasing_alpha}
				Let us consider a data-dependent level $\widetilde{\alpha}$, which claims significance at the 1\% level if $p \leq .01$, but otherwise claims significance at the 5\% level:
				\begin{align*}
					\widetilde{\alpha} = 
					\begin{cases}
						.01, \text{ if $p \leq .01$}, \\
						.05, \text{ if $p > .01$}.
					\end{cases}
				\end{align*}	
				This mimics a type of `$\alpha$-hacking', as $\widetilde{\alpha}$ is conveniently lowered to .01 if the $p$-value happens to fall below it.
				
				Suppose we use an exactly valid test function $\phi$, which means that $p \sim \textnormal{Unif}(0, 1]$.
				Then, the size at $a = .01$ equals
				\begin{align*}
					\mathbb{P}(p \leq \widetilde{\alpha} \mid \widetilde{\alpha} = .01) 
						= \mathbb{P}(p \leq .01 \mid p \leq .01) 
						= 1,
				\end{align*}
				since we only choose $\widetilde{\alpha} = .01$ if our $p$-value falls below $.01$.
				This size of $1$ is much larger than the selected level $a = .01$.
				Indeed, the size distortion at $a = .01$ equals $1 / .01 = 100$.
				
				At $a = .05$ the size equals 
				\begin{align*}
					\mathbb{P}(p \leq \widetilde{\alpha} \mid \widetilde{\alpha} = .05)
						= \mathbb{P}(p \leq .05 \mid p > .01)
						= 4 / 99
						\approx .04.
				\end{align*}
				Interestingly, the size here is roughly $.04$, which is lower than the selected level $a = .05$, so that the size distortion is below 1: $(4/99) / .05 = 80/99 \approx .81$.
			\end{exm}
			
		\subsection{Validity for a data-dependent level $\alpha$}\label{sec:max_expected}
			As seen in Example \ref{exm:decreasing_alpha}, the size distortion caused by using a data-dependent level $\widetilde{\alpha}$ may be above or below 1, depending on the realization of the data-dependent level $\widetilde{\alpha}$.
			In order to define a notion of validity of a test function $\phi$ or $p$-value $p$ for a data-dependent level $\widetilde{\alpha}$, we \emph{must} choose how we weigh these different distortions.
			This choice is overlooked in preceding work on testing with data-dependent $\widetilde{\alpha}$.
			
			While we discuss more options in Section \ref{sec:beyond_e-values}, we focus on two options here: controlling the maximum size distortion and the expected (relative) size distortion.
			Both these options are a generalization of the notion of size distortion for pre-specified levels $\alpha$.
			Indeed, if $\widetilde{\alpha}$ can only take on a single value $\alpha$, then both the maximum and expected size distortion reduce to \eqref{dfn:size_distortion}. \\
					
			\noindent\textbf{Maximum size distortion}.
			Perhaps the most obvious choice would be to demand that the size at $a$ is below $a$ for all $a$, or at least for all $a$ in the support of $\widetilde{\alpha}$.
			That is, we say that a test function $\phi$ or $p$-value is valid for a data-dependent level $\widetilde{\alpha}$ under the maximum size distortion if:
			\begin{align*}
				\mathbb{P}(\phi(\widetilde{\alpha}) = 1 \mid \widetilde{\alpha} = a) 
					\equiv \mathbb{P}(p \leq \widetilde{\alpha} \mid \widetilde{\alpha} = a) 
					\leq a,
			\end{align*}
			for all $a$ in the support of $\widetilde{\alpha}$.
			% The conservative approach is to say that the size of a test $\phi$ for a data-dependent level $\widetilde{\alpha}$ should be bounded by $a$ for each value $a$ it can take.
			Formulated in terms of the size distortion, this is equivalent to demanding that the \emph{maximum size distortion} over $a$ is at most 1:
			\begin{align*}
				\sup_{a \in \text{support}(\widetilde{\alpha})} \frac{\mathbb{P}(\phi(\widetilde{\alpha}) = 1 \mid \widetilde{\alpha} = a)}{a} \leq 1,
			\end{align*}
			where $\text{support}(\widetilde{\alpha})$ contains all the values $\widetilde{\alpha}$ may take.
			This is also known as its essential supremum. \\
			
			While control over the maximum size distortion may superficially seem ideal, it is too strict to develop a meaningful theory of testing with data-dependent $\alpha$.
			Indeed, it even bans clearly conservative choices of $\widetilde{\alpha}$ (see Example \ref{exm:conservative}).
			Moreover, it suffers from undesirable discontinuities (see Example \ref{exm:fragility}), due to the fact that the supremum over the support of a random variable is a very fragile operation: it can strongly depend on events with near-zero probability.
			For completeness, we study what post-hoc $\alpha$ hypothesis testing would look like with this condition as a starting point in Appendix \ref{apn:post-hoc_max}.
			There, we indeed find it is equivalent to requiring $p \geq 1$ or $\phi(\alpha) = 0$ for all $0 \leq \alpha \leq 1$, leaving no room for a useful theory. \\	
		
			\noindent\textbf{Expected size distortion}.
			As the maximum size distortion is problematic, we propose to define validity through the \emph{expected size distortion}.
			This permits some distortions, as long as they are controlled in expectation.
			
			\begin{dfn}[Validity for data-dependent level $\widetilde{\alpha}$]
				A test function $\phi$ or $p$-value $p$ is valid for a data-dependent level $\widetilde{\alpha}$, if the expected size distortion is at most 1:
				\begin{align*}
					\mathbb{E}_{\widetilde{\alpha}}\left[\frac{\mathbb{P}(\phi(\widetilde{\alpha}) = 1 \mid \widetilde{\alpha})}{\widetilde{\alpha}}\right]
					\equiv 
					\mathbb{E}_{\widetilde{\alpha}}\left[\frac{\mathbb{P}(p \leq \widetilde{\alpha} \mid \widetilde{\alpha})}{\widetilde{\alpha}}\right]
					\leq 1.
				\end{align*}
			\end{dfn}
			The interpretation is that if we use the data-dependent level $\widetilde{\alpha}$ many times, then on average the size distortion is below 1.
			
			The expected size distortion can actually be very concisely written as the expectation of $\phi(\widetilde{\alpha}) / \widetilde{\alpha} \equiv \mathbb{I}\{p \leq \widetilde{\alpha}\} / \widetilde{\alpha}$, as shown in Proposition \ref{prp:compact}.
			While less interpretable, this expression is mathematically more convenient to work with.
			
			\begin{prp}\label{prp:compact}
				We have
				\begin{align*}
					 \mathbb{E}_{\widetilde{\alpha}}\left[\frac{\mathbb{P}(\phi(\widetilde{\alpha}) = 1 \mid \widetilde{\alpha})}{\widetilde{\alpha}}\right]
					 	= \mathbb{E}\left[\frac{\phi(\widetilde{\alpha})}{\widetilde{\alpha}}\right]
					 	\equiv \mathbb{E}\left[\frac{\mathbb{I}\{p \leq \widetilde{\alpha}\}}{\widetilde{\alpha}}\right].
				\end{align*}
			\end{prp}
			\begin{proof}
				This follows from the fact that $\mathbb{P}(\phi(\widetilde{\alpha}) = 1 \mid \widetilde{\alpha})/\widetilde{\alpha} = \mathbb{E}(\phi(\widetilde{\alpha})\mid \widetilde{\alpha})/\widetilde{\alpha} = \mathbb{E}(\phi(\widetilde{\alpha})/\widetilde{\alpha}  \mid \widetilde{\alpha})$, and then taking the expectation over $\widetilde{\alpha}$.
			\end{proof}

		\subsection{Examples}
			We have constructed a variety of examples in which we illustrate the expected and maximum size distortion.
			\begin{exm}[A conservative data-dependent level $\widetilde{\alpha}$]\label{exm:conservative}
				In this example, we show that control of the maximum size distortion even bans conservative choices of the data-dependent level $\widetilde{\alpha}$.
				In particular, suppose that instead of a fixed $\alpha = .01$, we choose a data-dependent $\widetilde{\alpha}$ that reports a rejection at a conservatively large level of $.02$ whenever $p \leq .01$, and otherwise equals $0.01$:
				\begin{align*}
					\widetilde{\alpha}
						=
						\begin{cases}
							.02, &\text{ if } p \leq .01, \\
							.01, &\text{ if } p > .01.
						\end{cases}
				\end{align*}
				Regardless of the distribution of the $p$-value, the size at $.02$ equals $\mathbb{P}(p \leq \widetilde{\alpha} \mid \widetilde{\alpha} = .02) = 1$ and at $.01$ it equals $\mathbb{P}(p \leq \widetilde{\alpha} \mid \widetilde{\alpha} = .01) = 0$.
				As a consequence, the maximum size distortion is $1 / .02 = 50 \gg 1$, suggesting the $\widetilde{\alpha}$ is too liberal, despite the fact that this data-dependent $\widetilde{\alpha}$ is clearly more conservative than the fixed level $.01$.
				
				On the other hand, for an exact $p$-value, the expected size distortion equals $.01 \times 50 + .99 \times 0 = .5 \leq 1$, which properly expresses that this $\widetilde{\alpha}$ is conservative.
			\end{exm}
			
			\begin{exm}[Discontinuity of maximum size distortion]\label{exm:fragility}
				The point of this example is to show that even if a data-dependent level $\widetilde{\alpha}$ is `close to ok', the maximum size distortion may still be large.
				Suppose that, for a given constant $c > 0$, we choose the data-dependent level 
				\begin{align*}
					\widetilde{\alpha}_c
						=
						\begin{cases}
							c, &\text{ if } p \leq c, \\
							.05, &\text{ if } p > c.
						\end{cases}
				\end{align*}
				If $c = .05$, then we are in the data-independent setting with a fixed $\alpha = .05$.
				An exact $p$-value or test function is clearly valid for such a data-independent choice of $\alpha$.
				
				However, let us now consider choosing a constant $c < .05$ that is close to $.05$.
				Then, we find that the maximum size distortion is large.
				In particular, the limit as $c \nearrow .05$ does not coincide with the choice $c = .05$:
				\begin{align*}
					\lim_{c \nearrow .05} \sup_{a \in \textnormal{support}(\widetilde{\alpha}_c)}
					\frac{\mathbb{P}(\phi(\widetilde{\alpha}_c) = 1 \mid \widetilde{\alpha}_c = a)}{a}
						= 20,
				\end{align*}
				because $\mathbb{P}(\phi(\widetilde{\alpha}_c) = 1 \mid \widetilde{\alpha}_c = c) \equiv \mathbb{P}(p \leq c \mid p \leq c) = 1$, for all $c < .05$.
				This has to do with the fact that $\widetilde{\alpha}_c$ has two points of support for every $c < .05$, namely $c$ and $.05$.
				
				On the other hand, the limit of the expected size distortion as $c \nearrow .05$ equals 1: the same value as if we had set $c = .05$.
			\end{exm}
			
			\begin{exm}[Size distortion when $\alpha$-hacking]\label{exm:decreasing_alpha_continued}
				Continuing from Example \ref{exm:decreasing_alpha}, the maximum size distortion equals 100.
				However, we also see that the size distortion at $a = .05$ is below 1.
				This averages out to an expected size distortion of $.01 \times 100 + .99 \times 80/99 =  1.8$.		
				This shows that with an exactly valid test function, this type of $\alpha$-hacking-like $\widetilde{\alpha}$ is problematic both in terms of the maximum and expected size distortion.
			\end{exm}
			
			\begin{exm}[Rejecting at level $p$]\label{exm:rej_level_p}
				An extension of Example \ref{exm:decreasing_alpha} is the extreme form of $\alpha$-hacking, where we use the smallest level $\widetilde{\alpha}$ for which our test function rejects.
				That is, we claim rejection at level $\widetilde{\alpha} = p$.
				
				In this case we always reject, so that the size distortion equals $1 / a$, for each possible realization $a$ of $\widetilde{\alpha}$.
				As a consequence, both the maximum and expected size distortion are unbounded: $\sup_a 1/a = \infty$ and $\mathbb{E}\, 1/p = \log(1) - \log(0) = \infty$.
				This illustrates that `rejecting at level $p$' is indeed highly problematic with traditional valid $p$-values and test functions.
			\end{exm}
			
			\begin{exm}[Valid `$\alpha$-hacking']\label{exm:valid_hacking}
				The point of this example is to show that there exist settings in which a test function $\phi$ controls the expected but not the maximum size distortion.
				Suppose we again consider the $\widetilde{\alpha}$ from Example \ref{exm:decreasing_alpha}.
				But now we use a conservative test function $\phi(\alpha) = \mathbb{I}\{p \leq \alpha\}$ based on a valid $p$-value that is not uniform on $(0, 1]$:
				\begin{align*}
					p \sim
						\begin{cases}
							\textnormal{Unif}(0, 1), &\text{w.p. } 1/2, \\
							1, &\text{w.p. } 1/2.
						\end{cases}
				\end{align*}
				Then, the expected size distortion equals .9.
				This shows that by using a more conservative test function we can control the expected size distortion.
				
				On the other hand, the maximum size distortion equals 100 at $a = .01$, since $\mathbb{P}(\phi(\widetilde{\alpha}) = 1 \mid \widetilde{\alpha} = .01) = 1 > .01$.
			\end{exm}
		
		\section{Post-hoc $\alpha$ hypothesis testing}\label{sec:post-hoc_testing}
			In the previous section, we defined a notion of validity of a test function and $p$-value for a \emph{single} data-dependent level $\widetilde{\alpha}$.
			In this section, we discuss test functions and $p$-values that are valid for \emph{every} data-dependent choice of the level $\widetilde{\alpha}$.
			That is, the level $\alpha$ can be chosen \emph{post-hoc}.
			
			\begin{dfn}[Post-hoc $\alpha$ validity]
				We say that a test function $\phi$ or $p$-value is post-hoc ($\alpha$ valid) if its expected size distortion is at most 1 for every data-dependent level $\widetilde{\alpha}$:
				\begin{align}\label{dfn:pha_validity}
					\sup_{\widetilde{\alpha}}\mathbb{E}\left[\frac{\mathbb{P}(\phi(\widetilde{\alpha}) = 1 \mid \widetilde{\alpha})}{\widetilde{\alpha}}\right]
					\equiv \sup_{\widetilde{\alpha}}\mathbb{E}\left[\frac{\mathbb{P}(p \leq \widetilde{\alpha} \mid \widetilde{\alpha})}{\widetilde{\alpha}}\right]
					\leq 1,
				\end{align}
				where the supremum is over every data-dependent level (random variable) $\widetilde{\alpha}$.
			\end{dfn}
			
			If we use a post-hoc $\alpha$ valid test function $\phi$, then whatever data-dependent level $\widetilde{\alpha}$ we use, we still have the guarantee that the expected size distortion is at most 1 in expectation.
			
			Post-hoc $\alpha$ validity is a stronger guarantee than validity, as captured in Theorem \ref{thm:post-hoc_is_valid}.
			This means that if we use a post-hoc valid test function $\phi$, we also have the original guarantee that the test $\phi(\alpha)$ has a size smaller than $\alpha$ for any pre-specified level $\alpha$.
			\begin{thm}\label{thm:post-hoc_is_valid}
				A post-hoc valid test function is valid.
				Equivalently, a post-hoc $p$-value is also valid.
			\end{thm}
			\begin{proof}
				Note that \eqref{eq:valid_original} can be equivalently written as 
				\begin{align*}
					\sup_{\alpha} \mathbb{P}(\phi(\alpha) = 1) / \alpha
						= \sup_{\alpha} \mathbb{E}(\phi(\alpha) / \alpha)
						\leq 1.
				\end{align*}
				The result then follows from the observation that
				\begin{align*}
					\sup_{\alpha} \mathbb{E}(\phi(\alpha) / \alpha) 
						&\leq \mathbb{E}(\sup_{\alpha} \phi(\alpha) / \alpha)
						= \sup_{\widetilde{\alpha}} \mathbb{E}\left[\frac{\phi(\widetilde{\alpha})}{\widetilde{\alpha}}\right],
				\end{align*}
				and invoking Proposition \ref{prp:compact}.
			\end{proof}

		\subsection{Post-hoc $p$-values: a simplification}
			While the expression of the definition of post-hoc $\alpha$ validity in \eqref{dfn:pha_validity} seems complicated to work with, Theorem \ref{thm:php} yields an exceptionally clean representation of the post-hoc $p$-value.
			The key insight of the proof is that the $p$-value itself \emph{is} the smallest data-dependent level for which we reject.
			
			\begin{thm}\label{thm:php}
				$p$ is a post-hoc $p$-value if and only if
				\begin{align*}
					\mathbb{E}\, 1/p \leq 1.
				\end{align*}
			\end{thm}
			\begin{proof}
				By Proposition \ref{prp:compact},
				\begin{align*}
					\sup_{\widetilde{\alpha}}\mathbb{E}_{\widetilde{\alpha}}\left[\frac{\mathbb{P}(p \leq \widetilde{\alpha} \mid \widetilde{\alpha})}{\widetilde{\alpha}}\right] = 
					\sup_{\widetilde{\alpha}}\mathbb{E}\left[\frac{\mathbb{I}\{p \leq \widetilde{\alpha}\}}{\widetilde{\alpha}}\right].
				\end{align*}
				This latter term equals
				\begin{align*}
					\mathbb{E}\left[\sup_{\alpha}\frac{\mathbb{I}\{p \leq \alpha\}}{\alpha}\right],
				\end{align*}
				as the supremum in the expectation can arbitrarily depend on the data.
				Evaluating this supremum yields
				\begin{align*}
					\mathbb{E}\left[\frac{\mathbb{I}\{p \leq p\}}{p}\right]
						= \mathbb{E}\left[\frac{1}{p}\right].
				\end{align*}
			\end{proof}

			Theorem \ref{thm:php} shows that to find a post-hoc $p$-value, we merely need to find a non-negative random variable with expectation at most 1, and take its reciprocal.
			A post-hoc valid test function can then be recovered as $\phi(\alpha) = \mathbb{I}\{p \leq \alpha\}$.
			Conversely, starting with a post-hoc valid test function $\phi$, a post-hoc $p$-value can be constructed using $p = \inf\{\alpha : \phi(\alpha) = 1\}$.
			
			Post-hoc $p$-values inherit two useful merging properties from the definition of an expectation, which we capture in Proposition \ref{prp:prod_mult_non-random} and \ref{prp:harmonic_mult_non-random}.
			\begin{prp}\label{prp:prod_mult_non-random}
				A product of independent post-hoc $p$-values is a post-hoc $p$-value.
			\end{prp}
			\begin{prp}\label{prp:harmonic_mult_non-random}
				A weighted harmonic mean of post-hoc $p$-values is a post-hoc $p$-value.
			\end{prp}

			\begin{rmk}[Domain of post-hoc $p$-values]
				We should expect `useful' post-hoc $p$-values to sometimes take value above 1 under the null hypothesis.
				Indeed, as their reciprocal is at most 1 in expectation, a post-hoc $p$-value that is only supported on $(0, 1]$ must almost surely equal 1 under the null hypothesis.
				Of course, we want a post-hoc $p$-value to be much smaller than 1 under the alternative hypothesis.
				
				This may be unsettling at first sight as valid traditional $p$-values typically don't take values above 1.
				However, we stress that the $p$-value corresponds to the level, and not the size of the test.
				Indeed, the size of a test is at most 1 since it is a probability, but the level has no such restrictions.
			\end{rmk}

			\begin{rmk}[Non-decreasing test function without loss of generality]\label{rmk:non-decreasing}
				Note that post-hoc $\alpha$ validity only restricts the behavior of a test function $\phi$ at the smallest level at which we reject: the $p$-value.
				This means that in the context of post-hoc $\alpha$ validity, the restriction to non-decreasing test functions is effectively without loss of generality: if we intend to make the most powerful claim possible, then we might as well reject for larger levels, too.
				That is, every post-hoc valid test family that is not a non-decreasing test function is dominated by one that is.
			\end{rmk}

	\section{Post-hoc $\alpha$ power}\label{sec:power}
		\subsection{Generalizing traditional power?}
			It remains to discuss what should constitute a \emph{good} post-hoc $p$-value or post-hoc valid test function.
			Indeed, the constant 1 is also a post-hoc $p$-value, but completely uninformative.
			Instead, we would like to use a post-hoc $p$-value that is `small' if our hypothesis is false: if the alternative hypothesis holds.
			
			In traditional hypothesis testing, it is common to choose the test that maximizes the rejection probability (power) under the alternative hypothesis among tests that are valid at the pre-specified level $\alpha$.
			We can certainly also maximize the power at a pre-specified $\alpha$ under the restriction that the test function it comes from is post-hoc valid, which we explore in Section \ref{sec:classical_power}.
			However, we would preferably define a post-hoc notion of power that performs well at the data-dependent level $p$.
			Curiously, this is not possible because we always reject the hypothesis at the data-dependent level $p$, by definition of a $p$-value.
			Hence, we instead aim to `minimize' the $p$-value under the alternative, with the goal to provide the strongest possible certificate on our rejection decision.
			
		\subsection{Expected-utility}
			To describe optimal post-hoc $p$-values, we use the connection to $e$-values and leverage the ``Neyman-Pearson lemma for $e$-values" recently introduced in \citet{koning2024continuous}.
			
			In particular, we consider a utility function $U$ that defines our preference for each possible realization of the post-hoc $p$-value.
			To obtain the cleanest-looking results, we define the utility function on the $1/p$-scale, maximizing
			\begin{align*}
				\mathbb{E}^{\mathbb{Q}}\left[U(1/p)\right],
			\end{align*}
			over post-hoc valid $p$-values, where $\mathbb{Q}$ is some alternative distribution.
			For example, for $U(1/p) = \log(1/p) = -\log(p)$ we value a smaller $p$-value linearly in its order of magnitude: $p = 0.01$ has twice the utility of $p = 0.1$.
			
			In Theorem \ref{thm:optimality}, we show how the Neyman-Pearson lemma for $e$-values of \citet{koning2024continuous} connects to post-hoc testing, by using it to characterize expected-utility optimal post-hoc $p$-values.
			Here, we let $f_{\mathbb{P}}$ and $f_{\mathbb{Q}}$ denote the density under the null and alternative, respectively, with respect to some dominating measure (which always exists).
			
			\begin{thm}[\citet{koning2024continuous}]\label{thm:optimality}
				Let $U : [0, \infty] \to [-\infty, \infty]$ be upper-semicontinuous, concave and non-decreasing.
				If $p^*$ is optimal then 
				\begin{align}\label{eq:utility_optimal}
					\lambda \frac{f_{\mathbb{P}}}{f_{\mathbb{Q}}} \in \partial U\left(1/p^*\right), \quad \mathbb{Q}\textnormal{-almost surely},
				\end{align}
				for some normalization constant $\lambda \geq 0$.
			\end{thm}
			
			\begin{cor}[Differentiable utility]
				If $U$ is differentiable, then an optimizer $p^*$ satisfies
				\begin{align*}
					1/p^* = (U')^{-1}\left(\lambda f_{\mathbb{P}} / f_{\mathbb{Q}}\right), \quad \mathbb{Q}\textnormal{-almost surely},
				\end{align*}	
				where $(U')^{-1}(y) = \inf\{x \geq 0 : U'(x) \leq y\}$ is the generalized inverse.
			\end{cor}
			
			\begin{cor}[Log-optimal]\label{cor:log_optimal}
				If $U = \log$, then $p^* = f_{\mathbb{P}} / f_{\mathbb{Q}}$.
			\end{cor}

			A general condition for an optimizer to exist is that a normalization constant $\lambda$ exists such that a solution $p_\lambda$ to \eqref{eq:utility_optimal} satisfies $\mathbb{E}^{\mathbb{P}}[1/p_\lambda] = 1$ or $\lambda = 0$ \citep{koning2024continuous}.
			Some sufficient conditions are: (1) the sample space is finite, (2) $U'(x)x$ is bounded from above, (3) $\{U(1/p) : \mathbb{E}^{\mathbb{P}}[1/p] \leq 1\}$ is uniformly integrable, (4) $U$ is bounded from above (implies (2) if differentiable and (3)).
			
			\begin{exm}[Log-optimal post-hoc test function]\label{exm:log_optimal}
				Let us consider $U = \log$, so that $p^* = d\mathcal{N}(0, 1) / d\mathcal{N}(1, 1)$ by Corollary \ref{cor:log_optimal}.
				Its test function equals
				\begin{align}\label{php_LR_norm}
					\phi(\alpha) 
						=  \mathbb{I}\left\{\frac{d\mathcal{N}(0, 1)}{d\mathcal{N}(1, 1)}(X) \leq \alpha\right\}
						= \mathbb{I}&\left\{\frac{d\mathcal{N}(1, 1)}{d\mathcal{N}(0, 1)}(X) \geq 1/\alpha\right\}.
				\end{align}
				
				The representation in \eqref{php_LR_norm} permits a clean comparison to the likelihood ratio test.
				At any pre-specified level $\alpha$, this test function yields a lower power than the likelihood ratio test. 
				For example, at $\alpha = .05$, the critical value of the likelihood ratio roughly equals 3.14, whereas the same statistic is compared in \eqref{php_LR_norm} with the much larger critical value $1/\alpha = 20$.
				
				Of course, the benefit of a post-hoc $\alpha$ hypothesis test is that we need not pre-specify $\alpha$.
			\end{exm}
			
		\subsection{Nesting classical power}\label{sec:classical_power}
			An interesting question is what happens if we consider the post-hoc test function $\phi$ that maximizes the power at a pre-specified level $\alpha^*$.
			As we only care about its behavior at $\alpha^*$, its form is actually easy to predict.
			Indeed, given the most powerful test $\tau^*$ at level $\alpha^*$, we have $\phi(\alpha) = 0$ for $\alpha < \alpha^*$ and $\phi(\alpha) = \tau^*$ for $\alpha \geq \alpha^*$.
			The corresponding $p$-value is $p^* = \alpha^*$ if $\tau^* = 1$ and $p^* = \infty$ if $\tau^* = 0$.
			
			In a continuous setting, $\tau^*$ is the likelihood ratio test, according to the classical Neyman-Pearson lemma.
			In a general setting, this is only true if we admit randomized decisions, which we do not allow here (see Appendix \ref{sec:randomized} for randomized post-hoc testing).
			Surprisingly, we may still recover the classical Neyman-Pearson lemma by \emph{replacing randomization by rejection at a level larger than $\alpha^*$}, using the ``Neyman-Pearson utility function $U(x) = x \wedge 1/\alpha^*$.
			This utility function expresses that we do not care for rejections at levels smaller than $\alpha^*$.
			We cover this result in Corollary \ref{cor:NP}.
			
			\begin{cor}[Post-hoc level Neyman-Pearson]\label{cor:NP}
				If $U(x) = x \wedge 1/\alpha^*$ then an optimizer exists and satisfies
				\begin{align*}
					p^*
						=
						\begin{cases}
							\alpha^* &\textnormal{ if } f_{\mathbb{P}} / f_{\mathbb{Q}} < c_\alpha^*, \\
							k &\textnormal{ if } f_{\mathbb{P}} / f_{\mathbb{Q}} = c_\alpha^*, \\
							\infty &\textnormal{ if } f_{\mathbb{P}} / f_{\mathbb{Q}} > c_\alpha^*, \\
						\end{cases}
				\end{align*}
				for some constants $k \in [\alpha^*, \infty]$, $c_\alpha^* \geq 0$.
			\end{cor}

	\section{Connection to $e$-values}\label{sec:redefine_e-values}
		The reciprocal $1/p$ of a post-hoc $p$-value is a non-negative random variable with an expectation at most 1.
		A non-negative random variable with expectation at most 1 is also known as an $e$-value, which is a recently popularized statistical object \citep{howard2021time, shafer2021testing, vovk2021values, grunwald2023safe, ramdas2023gametheoretic}.
		
		\begin{dfn}[Original]\label{dfn:e-values}
			$e$ is a valid $e$-value if $\mathbb{E}\, e \leq 1$.
		\end{dfn}
		
		With $e$-values, it is standard to use the test function $\alpha \mapsto \mathbb{I}\{e \geq 1/\alpha\}$, which is valid due to Markov's inequality.
		For the post-hoc $p$-value $p = 1/e$, this test function simply equals $\alpha \mapsto \mathbb{I}\{p \leq \alpha\}$.
		Based on this link, our Theorem \ref{thm:php} shows this $e$-value-based test function is not just valid, but also post-hoc valid.
		Importantly, it also implies its converse: \emph{any post-hoc $\alpha$ valid test function is of the form} $\alpha \mapsto \mathbb{I}\{e \geq 1/\alpha\}$.
		
		A consequence is that the literature on $e$-values has unknowingly been studying post-hoc $\alpha$ hypothesis testing and post-hoc $p$-values.
		Indeed, many post-hoc $p$-values have already been developed, just thinly disguised by the operation $1/\cdot$.
						
	\section{Beyond e-values}\label{sec:beyond_e-values}
		This section is intended for readers who are interested in the abstract theory underlying post-hoc testing and $e$-values.
		In the preceding, we focused on a specific (arguably reasonable) way to measure the size distortion that leads to the $e$-value, mentioning other options in passing.
		In this section, we study a more general framework that nests such alternative choices, leading us beyond the $e$-value.
				
		\subsection{Abstract evidence}
			We start by abstractly quantifying evidence.
			For this purpose, we introduce an ``evidence space'' $\mathcal{D}$, which is a decision space that is totally ordered with respect to some binary order relation $\precsim$.
			Here, we interpret `larger' values as indicating more evidence than smaller values.
			We assume that $\mathcal{D}$ has a bottom element  $``0" := \inf \mathcal{D}$ and a top element $``\infty" := \sup \mathcal{D}$.
			
			In the context of classical hypothesis testing, we may think of $\mathcal{D}$ to contain the decisions to reject the hypothesis at different significance levels:
			\begin{align*}
				\mathcal{D}
					= \{\textnormal{not reject}, \textnormal{reject at } \alpha_1, \textnormal{reject at } \alpha_2, \dots\},
			\end{align*}
			where we couple the bottom element ``0'' to the decision not to reject.
			The order relation $\precsim$ then captures the idea that a rejection at a smaller level is a stronger claim: ``reject at level 0.05'' $\precsim$ ``reject at level 0.01''.
		
			Abstracting the notion of both a $p$-value and an $e$-value, we introduce an \emph{evidence variable}
			\begin{align*}
				\varepsilon : \mathcal{X} \to \mathcal{D},
			\end{align*}
			where $\mathcal{X}$ is our sample space.
			We abstractly define a test $\phi(d)$, $d \in \mathcal{D}$, as a `binary' evidence variable
			\begin{align*}
				\phi(d) : \mathcal{X} \to \{``0", d\},
			\end{align*}
			returning either no evidence ``0" or $d$ evidence.
			We defer discussions of measurability to Appendix \ref{sec:measurability_conditioning}.
		
		\subsection{Abstract post-hoc level testing}\label{sec:post-hoc_abstract}
			To discuss post-hoc testing, we assume $\mathcal{D}$ is (Dedekind) complete, meaning that every subset has a supremum and infimum in $\mathcal{D}$ element.
			For a given test function $\phi$, we then define its post-hoc evidence variable as $\varepsilon_\phi := \sup_{d \in \mathcal{D}} \phi(d)$.
			This returns the strongest evidence returned by any of the tests $\phi(d)$, $d \in \mathcal{D}$.
			
			This already yields the core result underlying post-hoc level testing.
			
			\begin{thm}\label{thm:post-hoc_abstract}
				For a given test function $\phi$, $\varepsilon_\phi = \sup_{d \in \mathcal{D}} \phi(d)$ is an evidence variable.
				Every evidence variable $\varepsilon$ is the post-hoc evidence variable of the test function
				\begin{align*}
					\phi_\varepsilon(d)
						:=
						\begin{cases}
							d, &\textnormal{ if } d \precsim \varepsilon, \\
							``0", &\textnormal{ otherwise}.
						\end{cases}
				\end{align*}
			\end{thm}
			\begin{proof}
				For the first claim, since $\mathcal{D}$ is Dedekind complete, $\varepsilon_\phi = \sup_{d \in \mathcal{D}} \phi(d)$ is $\mathcal{D}$-valued and hence an evidence variable.

				For the second claim, note that $\{\phi_\varepsilon(d) : d \in \mathcal{D}\} = \{d \in \mathcal{D} : d \precsim \varepsilon\}$.
				Hence, $\sup_{d \in \mathcal{D}} \phi_\varepsilon(d) = \sup\{d \in \mathcal{D} : d \precsim \varepsilon\} = \varepsilon$.
			\end{proof}
			
			\begin{rmk}
				The supremum of any collection of evidence variables is an evidence variable.	
			\end{rmk}

		\subsection{Certainty equivalence and validity}
			To define an abstract notion of validity, we introduce a certainty equivalent $\rho : \Delta \to \mathcal{D}$, where $\Delta = \mathcal{D}^{\mathcal{X}}$ denotes the space of evidence variables.
			Such a certainty equivalent may be interpreted to return an amount of evidence $d \in \mathcal{D}$ that is of `equivalent value' as $\varepsilon$.
			
			A first axiom we assume $\rho$ to satisfy is idempotence: if $\varepsilon_d \equiv d$ then $\rho(\varepsilon_d) = d$.
			This calibrates $\rho$ to certain outcomes.
			Moreover, we assume monotonicity:  $\varepsilon^-(x) \precsim \varepsilon^+(x)$ for every $x$ $\implies$ $\rho(\varepsilon^-) \precsim \rho(\varepsilon^+)$.
			This may be viewed as compatibility with $\precsim$.
			Combined, these two axioms are easily shown to already imply a notion of `internality': if $\mathcal{D}$ is complete then $\inf \varepsilon \precsim 	\rho(\varepsilon) \precsim \sup \varepsilon$.
			
			Using a certainty equivalent, we can define a notion of validity of an evidence variable $\varepsilon$ by comparing it to some benchmark $b \in \mathcal{D}$ amount of evidence:
			\begin{align*}
				\rho(\varepsilon) \precsim b.
			\end{align*}
			If we assume $\mathcal{D}$ has some identity element ``1'' and we equip it with an invertible binary operator $\times$ (extended pointwise to $\Delta$), then we can always normalize the benchmark $b$ and our evidence variable and define validity with respect to ``1":
			\begin{align}\label{eq:equivariance}
				\rho(\varepsilon) \precsim ``1".	
			\end{align}
			
			An interesting class of certainty equivalents that we consider in the remainder of this section are those based on the quasi-arithmetic mean \citep{kolmogorov1930notion, de2016mean}.\footnote{Predating his axiomatization of probability, \citet{kolmogorov1930notion} shows in a simplified setting that idempotence and monotonicity, alongside replacement and continuity axioms \emph{characterize} such certainty equivalents. In Appendix \ref{sec:measurability_conditioning}, we generalize the replacement axiom to total orders, showing that it is equivalent to assuming the existence of a `conditional' certainty equivalent given a sub-information structure. We leave characterizing a form such as \eqref{eq:ce_loss} to future work.}
			In particular, we use a monotone loss function $L : \mathcal{D} \to [-\infty, \infty]$ to define the certainty equivalent
			\begin{align}\label{eq:ce_loss}
				\rho(\varepsilon)
					= L^{-1}\left(\mathbb{E}\left[L(\varepsilon)\right]\right),
			\end{align}
			where $\mathbb{E}$ may be viewed as the expectation under the null hypothesis, and we assume $L(\mathcal{D})$ to be interval-valued to ensure this is well-defined.
			
			\begin{rmk}
				We can generalize the certainty equivalent $\rho_{\mathbb{P}}(\varepsilon) = L^{-1}\left(\mathbb{E}^{\mathbb{P}}[L(\varepsilon)]\right)$ to a composite hypothesis $H$ by using the certainty equivalent $\rho_H(\varepsilon) := \sup_{\mathbb{P} \in H} \rho_{\mathbb{P}}(\varepsilon)$.
			\end{rmk}
		
		\subsection{$e$-values, $p$-values and tests}
			To show that $e$-values, $p$-values and tests are all special evidence variables, we consider $\mathcal{D} = [0, \infty]$.
			
			Taking $\mathcal{D} = [0, \infty]$, the $e$-values and $p$-value now (superficially) diverge, depending on whether we (arbitrarily) choose to couple large or small numerical values to strong evidence.
			In particular, the $e$-value emerges if we identify $\precsim$ with $\leq$, couple $``0"$ to $0$ and $``\infty"$ to $\infty$.
			The $p$-value emerges if we instead consider its order dual, identifying $\precsim$ with $\geq$, and coupling $``0"$ to $\infty$ and $``\infty"$ to $0$.
			Equipping $[0, \infty]$ with scalar multiplication, we may swap between scales by taking the reciprocal map: $e = 1/p$, recovering the reciprocal duality between $e$-values and $p$-values from Section \ref{sec:redefine_e-values}.
			
			The classical validity condition of an $e$-value corresponds to taking the identity loss function $L : x \mapsto x$, and normalizing the current evidence to 1 so that\footnote{We may equivalently express this in terms of a post-hoc $p$-value, taking the loss $L : x \mapsto x^{-1}$ and identifying $\precsim$ with $\geq$: $\rho(p)
						= L^{-1}(\mathbb{E}[L(p)])
						= 1/\mathbb{E}[1/p]
						\geq 1$, which is equivalent to $\mathbb{E}[1/p] \leq 1$.}
			\begin{align*}
				\rho(e)
					= L^{-1}(\mathbb{E}[L(e)])
					= \mathbb{E}[e]
					\leq 1.
			\end{align*}
			Taking the loss $L(x) = 1/x$ and identifying $\precsim$ with $\geq$, this may be equivalently expressed in terms of a (post-hoc) $p$-value: $\rho(p) = L^{-1}(\mathbb{E}[L(p)]) = 1/\mathbb{E}[1/p] \geq 1$.
			These are the same underlying evidence variables, merely expressed using a different representation.
			
			A level $\alpha$ test $\phi(\alpha)$ is classically viewed as a $\{0, 1\}$-valued map, using $0$ for $``0"$, $\alpha$ as an evidence benchmark and $1$ for the decision to reject at level $\alpha$, and said to be valid if $\mathbb{E}[\phi(\alpha)] \leq \alpha$.
			To facilitate the comparison across levels, \citet{koning2024continuous} alternatively views level $\alpha$ tests $\phi(\alpha)$ as $\{0, 1/\alpha\}$-valued, so that $1/\alpha$ represents the decision to reject at level $\alpha$.
			This frees $1$ to be used as a universal benchmark, so that a test is valid if $\mathbb{E}[\phi(\alpha)] \leq 1$, highlighting that tests are nothing more than `binary' $e$-values.
			Alternatively, we could equivalently describe level $\alpha$ tests as $\{\alpha, \infty\}$-valued, making them valid if $1/\mathbb{E}[1/\phi(\alpha)] \geq 1$, showing that tests are also binary (post-hoc) $p$-values.
			
			Classical validity of a $p$-value may be viewed as instead bounding the certainty equivalent $\rho(p) = \sup_{\alpha }\mathbb{E}[\mathbb{I}\{p \leq \alpha\}/\alpha]$ by 1.
			In the context of the above, this can be viewed as the validity of the family $\phi$ of evidence variables $\phi(\alpha) = \mathbb{I}\{p \leq \alpha\}/\alpha$, rather than validity of $p$ itself.
					
		\subsection{Generalized-mean validity}\label{sec:generalized_mean}
			In this section, we generalize beyond the classical $e$-value and post-hoc $p$-value using different certainty equivalents than $\rho = \mathbb{E}$.
			In particular, we focus on the class of equivariant certainty equivalents on $\mathcal{D} = [0, \infty]$ under multiplication.
			Equivariance is the property that permits us to normalize the benchmark to 1, as in \eqref{eq:equivariance}.
			
			The class of equivariant quasi-arithmetic means on $[0, \infty]$ is the $h$-generalized means \citep{hardy1934inequalities}.
			The $h$-generalized mean (henceforth `$h$-mean') $\rho_h$ of an evidence variable $\varepsilon$ is defined as
			\begin{align*}
				\rho_h(\varepsilon)
					=
					\begin{cases}
						\left(\mathbb{E}\left[\left(\varepsilon \right)^h\right]\right)^{1/h}, &\text{ if } h \neq 0, |h| < \infty, \\
						\exp\left(\mathbb{E}\left[\log\left(\varepsilon \right)\right]\right), &\text{ if } h = 0, \\
						\esssup \varepsilon, &\text{ if } h = \infty, \\
						\essinf \varepsilon, &\text{ if } h = -\infty, \\
					\end{cases}
			\end{align*}
			where $\esssup$ and $\essinf$ return the supremum and infimum of the support of a random variable.
			This may be seen as using the loss $L(x) = x^h$ and taking limits for the $-\infty$, $0$ and $\infty$ cases.
			The choices $h = 1$, $0$ and $-1$ yield the arithmetic, geometric and harmonic mean, respectively, which are jointly known as the Pythagorean means.
			
			Note that we may apply $\rho_h$ to both $e$-values and $p$-values, as $\rho_h(e) \leq 1$ is equivalent to $\rho_{-h}(p) \geq 1$ for $p = 1/e$.
			But to avoid ambiguity in the remainder of this section, we focus on the $e$-value scale.
			
			Using the equivariance of $\rho_h$ to normalize the benchmark evidence to 1, we generalize the validity of the classical $e$-value in Definition \ref{dfn:h-valid}.
			The classical $e$-value is recovered for $h = 1$.
			\begin{dfn}[$h$-validity]\label{dfn:h-valid}
				We say that an $e$-value is $h$-valid if $\rho_h(e) \leq 1$.
			\end{dfn}
			
			We now discuss several properties of $h$-valid $e$-values.
			We start by relating the validity conditions for different choices of $h$.
			\begin{prp}\label{prp:gen_mean_inequality}
				Let $h^+ \geq h^-$.
				If $e$ is $h^+$ valid, then it is $h^-$ valid.
				If $e$ is not $h^-$ valid, then it is not $h^+$ valid.
			\end{prp}
			\begin{proof}
				This follows directly from the generalized-means inequality: $\rho_{h^-}(e) \leq \rho_{h^+}(e)$.
			\end{proof}
			
			Proposition \ref{prp:minimal_h} shows that $h = 1$ is the smallest value of $h$ for which $h$-validity ensures that the induced test family $\phi_e$ is (classically) valid.
			Equivalently, within the $h$-mean class, the condition $E[1/p] \leq 1$ is the weakest condition that still guarantees that the $p$-value is valid in the classical sense.
			We stress that this does not disqualify them to be used as evidence variables; they merely satisfy a weaker validity condition.
			Its proof is found in Appendix \ref{proof:minimal_h}.
			
			\begin{prp}\label{prp:minimal_h}
				Let $h \leq 1$.
				If $e$ is $1$-valid, then the test function $\phi_e$ is valid.
				Conversely, for every $h < 1$ there exists an $e$-value $e$ that is $h$-valid but for which the test function $\phi_e$ is not valid.
			\end{prp}
			
			The following two results show how the merging properties of $e$-values generalize to $h$-valid $e$-values.
			We omit the proofs, as they are simple but take up a considerable amount of space.
			\begin{prp}
				A (weighted) $h$-mean of $h$-valid $e$-values is an $h$-valid $e$-value.
			\end{prp}
%			\begin{proof}
%				Assuming $\sum_{i=1}^n w_i  = 1$ and $w_i \geq 0$, we have
%				\begin{align*}
%					&\left(\mathbb{E}\left[\left\{\left(\sum_i w_i(e_i)^h\right\}^{1/h}\right)^h\right]\right)^{1/h} \\
%						&= \left(\mathbb{E}\left[\sum_i w_i(e_i)^h\right]\right)^{1/h}
%						= \left(\sum_i w_i \mathbb{E}\left[(e_i)^h\right]\right)^{1/h} \\
%						&= \left(\sum_i w_i \left(\left\{\mathbb{E}\left[(e_i)^h\right]\right\}^{1/h}\right)^h\right)^{1/h}
%						\leq \left(\sum_i w_i( 1)^h\right)^{1/h}
%						= 1,
%				\end{align*}
%				where the inequality follows from the monotonicity of the $h$-mean.
%			\end{proof}

			\begin{prp}
				The product of independent $h$-valid $e$-values is $h$-valid.
			\end{prp}
%			\begin{proof}
%				\begin{align*}
%					\left(\mathbb{E}\left[\left(\prod_{i=1}^n e_i\right)^h\right]\right)^{1/h}
%					= \prod_{i=1}^n\left(\mathbb{E}\left[\left(e_i\right)^h\right]\right)^{1/h}
%					\leq 1,
%				\end{align*}	
%				where the equality follows from the fact that powers distribute over products, $(xy)^h = x^hy^h$,  and the fact that the expectation of the product of independent random variables is the product of their expectations.
%			\end{proof}

			\begin{rmk}[Sequential testing and equivariance]
				The class of $h$-valid $e$-values seems particularly useful in sequential settings, as the equivariance permits sequential re-normalization of the current evidence to $1$.	
			\end{rmk}

		\subsection{Examples}
			
			\begin{exm}[Harmonic $e$-values]\label{exm:harmonic}
				We say $e$ is a harmonic $e$-value if it is $h = -1$ valid: $\rho_{-1}(e) \equiv (\mathbb{E}[e^{-1}])^{-1} \leq 1$.
				
				Harmonic $e$-values have a surprising connection to Remark \ref{rmk:discrepancy}, where we consider the option to measure the size distortion through the \emph{size difference}:
				\begin{align*}
					\mathbb{P}(\phi(\widetilde{\alpha}) = \textnormal{reject at } \widetilde{\alpha} \mid \widetilde{\alpha} = a) - a.
				\end{align*}
				Bounding this by $0$ for all $a$ is equivalent to bounding the maximum size distortion.
				However, bounding the expected size difference leads to another notion of post-hoc validity:
				\begin{align}\label{eq:post-hoc_absolute_validity}
					\sup_{\widetilde{\alpha}}\mathbb{E}_{\widetilde{\alpha}}[\mathbb{P}(\phi(\widetilde{\alpha}) = \textnormal{reject at } \widetilde{\alpha} \mid \widetilde{\alpha}) - \widetilde{\alpha}] \leq 0,
				\end{align}
				where we restrict $\widetilde{\alpha}$ to those that lead to a rejection.
				The following result captures the connection to harmonic $e$-values.
				Its proof is found in Appendix \ref{proof:harmonic_size_difference}.
				
				\begin{prp}\label{prp:harmonic_size_difference}
					The test function $\phi$ is post-hoc valid in the sense of \eqref{eq:post-hoc_absolute_validity} if and only if its $e$-value $e$ is harmonic.
				\end{prp}
			\end{exm}
			
			\begin{exm}[Geometric $e$-values]
				We say $e$ is a geometric $e$-value if $\rho_0(e) \equiv \exp\{\mathbb{E} \log e\} \leq 1$.
				A remarkable property of geometric $e$-values is that one may use multiplication to merge both arbitrarily dependent and independent geometric $e$-values, as shown in Proposition \ref{prp:geometric_merging}.
				
				While maximizing the geometric expectation of an $e$-value under the alternative hypothesis has been widely considered in the $e$-value literature, we believe we are the first to consider it as a notion of validity and to describe this property.
				
				\begin{prp}\label{prp:geometric_merging}
					The product of geometric $e$-values is a geometric $e$-value	
				\end{prp}
				\begin{proof}
					Let $e_1, \dots, e_n$ be $n$ geometric $e$-values.
					Then,
					\begin{align*}
						\exp\left\{\mathbb{E} \log \prod_{i=1}^n e_i \right\}
							= \prod_{i=1}^n\exp\left\{\mathbb{E} \log  e_i \right\}
							\leq 1.
					\end{align*}
				\end{proof}
			\end{exm}
			
			\begin{exm}[Max $e$-values]\label{exm:max_e-values}
				We say that an $e$-value is a max $e$-value if $\rho_\infty(e) = \esssup e \leq 1$.
				Written in terms of the $p$-value, this corresponds to $\essinf p \geq 1$, which coincides with the notion of maximum size distortion discussed in Section \ref{sec:max_expected} and covered in detail in Appendix \ref{apn:post-hoc_max}.
			\end{exm}

	\section{Markov's equality}\label{sec:markov}
		The original inspiration for this work came from an attempt to close the gap in Markov's inequality, resulting in a kind of ``Markov's equality''.
		We believe this equality may be of independent interest, so we dedicate a section to this topic.
		
		To the best of our knowledge, all concentration inequalities (implicitly) rely on Markov's inequality in some step of their derivation.
		Replacing Markov's inequality by Markov's equality may pave the way to `concentration equalities'.
		We showcase this in an application to Ville's inequality.

		\subsection{Deterministic inequalities}
			Let us assume throughout that $X$ is some integrable non-negative random variable, $X \geq 0$.
			To start, let us consider the following simple inequalities,
			\begin{align}\label{ineq:deterministic_markov}
				\lfloor cX \wedge 1 \rfloor 
					\leq cX \wedge 1 
					\leq cX.
			\end{align}
			These inequalities follow from the simple fact that taking the minimum with 1 and rounding down are non-decreasing operations.
			
			We enjoy referring to these inequalities as the ``deterministic Markov's inequalities'', as variants of Markov's equality follow directly from applying the expectation operator to each term, as described in Lemma \ref{lem:MRMW-inequalities}.
			Indeed, Markov's inequality compares the first and final terms
			\begin{align*}
				\mathbb{P}_X(X \geq 1/c) 
					\leq c \mathbb{E}_X[X],
			\end{align*}
			the randomized Markov's inequality of \citet{ramdas2023randomized} compares the second and third
			\begin{align*}
				\mathbb{P}_{X, U}(X \geq U / c)
					\leq c \mathbb{E}_X[X],
			\end{align*}
			and the inequality comparing the first and second term
			\begin{align*}
				\mathbb{P}_X(X \geq 1/c)
					\leq \mathbb{E}_X[cX \wedge 1],
			\end{align*}
			can be viewed as a tighter non-integrable Markov's inequality, which is (implicitly) studied by \citet{wang2023extended} in a sequential setting.
						
			\begin{lem}[Markov-Ramdas-Manole-Wang inequalities]\label{lem:MRMW-inequalities}
				\begin{align*}
					\mathbb{P}_X(X \geq 1/c) \leq \mathbb{P}_{X, U}(X \geq U / c) \leq c \mathbb{E}_X[X].
				\end{align*}	
			\end{lem}
			\begin{proof}
				This follows from observing 	$\lfloor cX \wedge 1 \rfloor = \mathbb{I}\{X \geq 1/c\}$ and $cX \wedge 1 = \mathbb{P}_U(X \geq U/c)$, for $U \sim \textnormal{Unif}[0, 1]$ independent of $X$, and applying the expectation operator $\mathbb{E}_X$ to each term.
			\end{proof}
			
		\subsection{Markov's equality}
			Starting with the deterministic Markov's inequalities \eqref{ineq:deterministic_markov}, we obtain an equality if we divide each term by $c$ and take the supremum over $c$.
			
			\begin{lem}[Deterministic Markov's Equality]
				We have
				\begin{align}\label{eq:deterministic_markov}
					\sup_{c > 0} \lfloor cX \wedge 1 \rfloor /c
						= \sup_{c > 0}  X \wedge 1/c 
						= X.
				\end{align}
			\end{lem}
			\begin{proof}
				Define $b = 1/c$.
				The outer terms are equal because $\lfloor cX \wedge 1 \rfloor = \mathbb{I}\{X \geq 1/c\}$ and
				\begin{align*}
					\sup_{c > 0} \mathbb{I}\{X \geq 1/c\} /c
						= \sup_{b > 0} \mathbb{I}\{X \geq b\} b 
						= \sup_{0 < b \leq X} b
						= X.
				\end{align*}
				The inner term is squeezed to equality by the outer terms.
			\end{proof}

			By applying the expectation operator to the outer terms of \eqref{eq:deterministic_markov}, we obtain our ``Markov's equality'':
			\begin{align*}
				\mathbb{E}[\sup_{c > 0} \mathbb{I}\{X \geq 1/c\} /c] = \mathbb{E}[X].
			\end{align*}
			Markov's inequality follows from $\sup_{c > 0} \mathbb{E}[\cdot] \leq \mathbb{E}[\sup_{c > 0} \cdot]$.
			In Proposition \ref{prp:markov_equality}, we present a reformulation of this equality by using Proposition \ref{prp:compact}.
			\begin{prp}[Markov's Equality]\label{prp:markov_equality}
				\begin{align*}
					\sup_{\widetilde{c} > 0} \mathbb{E}_{\widetilde{c}}\left[\frac{\mathbb{P}(X \geq 1/\widetilde{c} \mid \widetilde{c})}{\widetilde{c}}\right]
						= \mathbb{E}[X].
				\end{align*}
			\end{prp}

			\begin{rmk}
				Analogous results may be obtained by replacing the expectation operator with some other order-preserving operation.
				One such example is the supremum over expectations $\sup_{\mathbb{P} \in H} \mathbb{E}^{\mathbb{P}}$ over some collection $H$ of probabilities, which is frequently encountered in hypothesis testing.
			\end{rmk}

		\subsection{Application: Ville's equality}
			Ville's inequality is a concentration inequality for martingale-like non-negative stochastic processes $(M_t)_{t \geq 0}$:
			\begin{align*}
				\sup_{\alpha} \sup_{\tau} \frac{\mathbb{P}(M_\tau \geq 1/\alpha)}{\alpha} \leq \mathbb{E}[M_0],
			\end{align*}
			where the supremum ranges over all stopping times $\tau$ adapted to the same filtration as $(M_t)_{t \geq 0}$, under some conditions on the stochastic process $(M_t)_{t \geq 0}$.
			It is often applied as a tool to derive other concentration inequalities, and it has also played a central role in the recent renaissance in sequential testing.
			
			In Proposition \ref{prp:ville_equality}, we showcase Markov's equality by deriving versions of ``Ville's equality'' for martingales, supermartingales and e-processes.
			Ville's inequality for such processes follows as a corollary, by restricting the supremum to data-independent $\alpha$.
			
			\begin{prp}[Ville's equalities]\label{prp:ville_equality}
				Let $(M_t)_{t \geq 0}$ be a non-negative stochastic process adapted to some filtration.
				If $(M_t)_{t \geq 0}$ is a martingale, then 
				\begin{align*}
					\sup_{\widetilde{\alpha}}\ &\mathbb{E}_{\widetilde{\alpha}}\left[\frac{\mathbb{P}(M_\tau \geq 1/\widetilde{\alpha}\mid \widetilde{\alpha})}{\widetilde{\alpha}}\right] 
						= \mathbb{E}[M_0],
				\end{align*}
				for every bounded stopping time $\tau$ adapted to the filtration.
				If $(M_t)_{t \geq 0}$ is a supermartingale, then 
				\begin{align*}
					\sup_{\widetilde{\alpha}} \sup_{\tau}\ &\mathbb{E}_{\widetilde{\alpha}}\left[\frac{\mathbb{P}(M_\tau \geq 1/\widetilde{\alpha}\mid \widetilde{\alpha})}{\widetilde{\alpha}}\right] 
						= \mathbb{E}[M_0],
				\end{align*}
				where the stopping time $\tau$ ranges over all stopping times adapted to the filtration.
				If $(M_t)_{t \geq 0}$ is an exact $e$-process starting at $M_0$ with respect to a possibly composite hypothesis $H$, then
				\begin{align*}
					\sup_{\widetilde{\alpha}} \sup_{\tau} \sup_{\mathbb{P} \in H} \ &\mathbb{E}_{\widetilde{\alpha}}^{\mathbb{P}}\left[\frac{\mathbb{P}(M_\tau \geq 1/\widetilde{\alpha}\mid \widetilde{\alpha})}{\widetilde{\alpha}}\right] 
						= \sup_{\mathbb{P} \in H} \mathbb{E}^{\mathbb{P}}[M_0].
				\end{align*}
			\end{prp}
			\begin{proof}
				Doob's optional stopping theorem for non-negative martingales gives $\mathbb{E}\left[M_\tau\right] = \mathbb{E}\left[M_0\right]$, for every bounded stopping time.
				Applying Markov's equality to $M_\tau$ then yields
				\begin{align}\label{eq:markov_equality_stochastic_process}
					\sup_{\widetilde{\alpha}} \mathbb{E}_{\widetilde{\alpha}}\left[\frac{\mathbb{P}(M_\tau \geq 1/\widetilde{\alpha}\mid \widetilde{\alpha})}{\widetilde{\alpha}}\right]
								= \mathbb{E}\left[M_\tau\right],
				\end{align}
				which yields the first result.
				For non-negative supermartingales we instead have $\mathbb{E}\left[M_\tau\right] \leq \mathbb{E}\left[M_0\right]$ for every stopping time.
				The second claim is then obtained by observing that the stopping time $\tau \equiv 0$ attains equality so that $\sup_{\tau} \mathbb{E}\left[M_\tau\right] = \mathbb{E}\left[M_0\right]$, and subsequently applying Markov's equality.
				The final claim follows directly from applying Markov's equality to the definition of an exact $e$-process: $\sup_{\tau} \sup_{\mathbb{P} \in H} \mathbb{E}^{\mathbb{P}}[M_\tau] = \sup_{\mathbb{P} \in H} \mathbb{E}^{\mathbb{P}}[M_0]$.
			\end{proof}

	\section{Acknowledgements}
		We thank the participants of the joint CWI / CMU ``e-readers'' seminar and the International Seminar on Selective Inference for their input. In particular, we thank Aaditya Ramdas, Chen Zhou, Glenn Shafer, Jesse Hemerik, Patrick Forr\'e, Peter Gr\"unwald, Ruben van Beesten, Ruodu Wang and Sam van Meer for fruitful discussions.

	\bibliographystyle{abbrvnat}
	\bibliography{Bibliography-MM-MC}
					
	\clearpage
	\appendix
	
	\section{Connection to \citet{grunwald2022beyond}}\label{sec:grunwald}
		In this section, we contrast our work to \citet{grunwald2022beyond}.
		Written in our notation from Section \ref{sec:beyond_e-values}, we specialize Proposition 1 of \citet{grunwald2022beyond} to the testing setting in Theorem \ref{thm:grunwald}, which describes the link he establishes between post-hoc testing and $e$-values.
		Instead of using a separate evidence space for each test $\phi(\alpha)$, we use the joint evidence space $\mathcal{D} = \{\textnormal{not reject}, \textnormal{reject at level }\alpha_1, \textnormal{reject at level }\alpha_2, \dots\}$.
		This joint evidence space allows us to write his result in this setting using a single loss function $\ell : \mathcal{D} \to [0, \infty]$, not to be confused with the loss $L$ from the quasi-arithmetic mean.
		
		\begin{thm}[\citet{grunwald2022beyond}]\label{thm:grunwald}
			We say that $\phi$ is \emph{type-I risk safe} if
			\begin{align}\label{eq:grunwald}
				\mathbb{E}\left[\sup_{\alpha} \ell(\phi(\alpha))\right] \leq 1.
			\end{align}
			Then, we have that $\phi$ is type-I risk safe if and only if $\sup_{\alpha} \ell(\phi(\alpha))$ is a valid $e$-value.
		\end{thm}
		
		While superficially related to our work, this result is of a different nature.
		In particular, \citet{grunwald2022beyond} treats $e$-values as \emph{loss-valued}: living on the same scale as the loss $\ell$.
		In our framework, $e$-values are \emph{decision-valued}: $\mathcal{D}$-valued.
		
		An advantage our decision-theoretic definition of evidence is that it covers perfectly reasonable variants of post-hoc testing that do not involve classical $e$-values.
		Indeed, in Example \ref{exm:harmonic} we find that post-hoc testing under the size difference of Remark \ref{rmk:discrepancy} naturally leads to evidence variables that satisfy $(\mathbb{E}[e^{-1}])^{-1} \leq 1$, not using classically valid $e$-values.
		
		Another notable difference is that \citet{grunwald2022beyond} does not show the necessity of $e$-values to post-hoc testing, while we show any (non-dominated) and $\{0, 1\}$-valued post-hoc test function is of the form $\phi(\alpha) = \mathbb{I}\{e \geq 1/\alpha\}$.
		We believe that this is also the reason that misses the duality between $e$-values and $p$-values.
		
		We suspect the underlying reason that the framework of \citet{grunwald2022beyond} does not recover these ideas is due to the adoption of the usual non-decision-theoretic definition of the $e$-value inside an otherwise decision-theoretic framework.
		Using our decision-theoretic definition of the $e$-value of a test function $\phi$ from Section \ref{sec:post-hoc_abstract}, and making the mild assumption that $\ell$ is non-decreasing and lower semicontinuous in the order topology, \eqref{eq:grunwald} becomes
		\begin{align*}
			\mathbb{E}\left[\ell(\sup_{\alpha} \phi(\alpha))\right]
				\equiv \mathbb{E}\left[\ell(e_\phi)\right]
				\leq 1,
		\end{align*}
		so that $e_\phi$ is generally \emph{not} a classically valid $e$-value, unless $\ell(x) = x$.
		
%		Perhaps surprisingly, \citet{grunwald2022beyond} does not observe the necessity of allowing $\alpha > 1$.
%		This is not surprising, because this does not naturally appear in his framework.
%		Indeed, suppose we take the loss function $\ell(\textnormal{reject at level } \alpha) = 1/\alpha$ and restrict $\alpha \in (0, 1)$, with $\ell(\textnormal{not reject}) = 0$.
%		A straightforward computation now yields
%		\begin{align*}
%			\sup_{\alpha \in (0, 1)} \ell(\phi(\alpha)) 
%				&= \sup_{\alpha \in (0, 1)} \mathbb{I}\{\phi(\alpha) = \textnormal{reject at level } \alpha\}/\alpha \\
%				&= \mathbb{I}\{p \leq 1\} / p.
%		\end{align*}
%		Hence, Theorem \ref{thm:grunwald} then states that $\phi$ is type-I risk safe if and only if
%		\begin{align*}
%			\mathbb{E}\left[\mathbb{I}\{p \leq 1\} / p\right] \leq 1.
%		\end{align*}
%		That is: if $\mathbb{I}\{p \leq 1\} / p$ is a (classically) valid $e$-value.
%		For this, it is sufficient that $\mathbb{E}[1/p] \leq 1$ since $\mathbb{I}\{p \leq 1\} / p \leq 1/p$, but it is certainly not necessary.
%		For example, if $\mathbb{P}(p = 1/2) = \mathbb{P}(p = 2) = 1/2$, then
%		$\mathbb{E}[\mathbb{I}\{p \leq 1\}/p] = 1$ but $\mathbb{E}[1/p] = 5/4 > 1$.
		
%		If we were to now restrict $p \in (0, 1]$, then this does recover $\mathbb{I}\{p \leq 1\}/p = 1/p$.
%		However, this also means $1/p \geq 1$, so that $\mathbb{E}[1/p] \leq 1$ can only hold if $p = 1$ almost surely, meaning that no nontrivial test function survives.

	\section{Double post-hoc $\alpha$ validity}\label{sec:two-way}
		In this section, we observe that we can have a $p$-value $p$ that is post-hoc if our hypothesis is true, with a reciprocal $1/p$ that is post-hoc if the hypothesis is false.
		Such a $p$-value is interesting, because it is a single number that has a guarantee in both directions.
		
		\begin{dfn}
			A $p$-value is double post-hoc if it is post-hoc under the hypothesis, and its reciprocal is post-hoc under the alternative.
		\end{dfn}
		
		A natural example of a double post-hoc $p$-value is a likelihood ratio, as formalized in Theorem \ref{thm:LR}.
		Likelihood ratios were already extensively considered in the context of $e$-values, as they maximize the geometric mean under the alternative: see e.g. \citet{koolen2022log} \citet{grunwald2023safe} and \citet{larsson2024numeraire}.
		Theorem \ref{thm:LR} gives an alternative motivation for looking at likelihood ratios in the context of $e$-values.
		The result generalizes to the composite likelihood ratios studied by \citet{larsson2024numeraire}.
		
		The result means that we can simultaneously interpret a likelihood ratio and its reciprocal as a post-hoc $p$-value for the hypothesis and the alternative.
		Such a double interpretation was already argued for likelihood ratios in the context of Bayes factors \citep{jeffreys1935some,jeffreys1998theory,kass1995bayes}.
		In this setting, Theorem \ref{thm:LR} implies that we can interpret such a Bayes factor as a double post-hoc $p$-value.
		\begin{thm}\label{thm:LR}
			Consider the hypothesis $\{\mathbb{P}_0\}$ and alternative $\{\mathbb{P}_1\}$.
			If $\mathbb{P}_0$ and $\mathbb{P}_1$ are mutually absolutely continuous, then the likelihood ratio $d\mathbb{P}_0/d\mathbb{P}_1$ is a double post-hoc $p$-value.	
		\end{thm}
		\begin{proof}
			The result follows from Theorem \ref{thm:php}, as $\mathbb{E}^{\mathbb{P}_0} d\mathbb{P}_1 / d\mathbb{P}_0 = \mathbb{E}^{\mathbb{P}_1} d\mathbb{P}_0 / d\mathbb{P}_1 = 1$. 
		\end{proof}
%		
%		Moreover, the following simple result shows that any two-way post-hoc $p$-value must be non-randomized.
%		\begin{prp}
%			Any two-way post-hoc $p$-function is non-randomized.
%		\end{prp}
%		\begin{proof}
%			A $p$-function $(p(u))_{u \in (0, 1]}$ is non-decreasing in $u$.
%			If its reciprocal needs to be a $p$-function as well, it must also be non-increasing.
%			Hence, it must be constant in $u$: $p(u) = p(1)$, for all $u \in (0, 1]$.
%		\end{proof}
		
		\begin{rmk}
			Note that the constant 1 is also a double post-hoc $p$-value, so that double post-hoc validity is not sufficient for a post-hoc $p$-value to be useful.
			It is merely an interesting additional guarantee that we may want a post-hoc $p$-value to satisfy.
		\end{rmk}

	\section{Post-hoc Multiple testing}
		\subsection{Post-hoc anytime validity}\label{sec:anytime}
			Post-hoc hypothesis testing naturally generalizes to a sequential setting, where we want to test while new data is still arriving.
			To model this, we consider a process of test functions $(\phi_t)_{t \in \mathbb{N}}$ indexed by a time $t$.
			This can be interpreted as observing test functions $\phi_1, \phi_2, \dots$ over time.
			
			To enforce the directionality of time, we additionally consider a collection of stopping times $\mathcal{T}$ with respect to an underlying filtration that specifies the available information at each moment in time.
			A process of test functions $(\phi_t)_{t \in \mathbb{N}}$ is said to be anytime valid with respect to $\mathcal{T}$ if
			\begin{align*}
				\sup_\alpha \sup_{\tau \in \mathcal{T}} \mathbb{E}\,\phi_\tau(\alpha)/\alpha \leq 1.
			\end{align*}
			Such a test process induces a $p$-process $(p_t)_{t \in \mathbb{N}}$, defined at each $t$ as the smallest value of $\alpha$ for which $\phi_t(\alpha) = 1$.
			We say that this $p$-process is valid if its underlying test process is valid, which coincides with the definition of a $p$-process given by \citet{ramdas2022admissible}. 
			
			Generalizing to post-hoc anytime validity, we say that a test process is post-hoc anytime valid if
			\begin{align*}
				\sup_{\tau \in \mathcal{T}} \mathbb{E}\,\sup_{\alpha > 0}\phi_\tau(\alpha)/\alpha \leq 1.
			\end{align*}
			Written in terms of the expected size distortion, this is equal to
			\begin{align*}
				\sup_{\widetilde{\alpha} }\sup_{\tau \in \mathcal{T}} \mathbb{E}\left[\frac{\mathbb{P}(\phi_\tau (\widetilde{\alpha}) = 1\mid \widetilde{\alpha})}{\widetilde{\alpha}}\right]\leq 1.
			\end{align*}
			
			We say that a $p$-process is post-hoc anytime valid if its underlying test process is.
			Equivalently, we can directly define post-hoc anytime validity through a $p$-process.
			\begin{dfn}
				A $p$-process is post-hoc anytime valid with respect to $\mathcal{T}$ if 
				\begin{align*}
					\sup_{\tau \in \mathcal{T}} \mathbb{E}\,1/p_\tau
						\leq 1.
				\end{align*}
			\end{dfn}
			
			This is precisely the reciprocal of an $e$-process \citep{ramdas2022admissible, ramdas2023gametheoretic}.
			These definitions generalize to the continuous time setting.
			
		\subsection{Post-hoc familywise error rate}\label{sec:fwer}
			Dropping the directionality of time in Section \ref{sec:anytime} yields the so-called familywise error rate \citep{ramdas2022admissible}.
			Let us now write $(\phi_{i})_{i \in \mathcal{I}}$ for the family of test functions, where $\mathcal{I}$ is some index set, where we use $i$ instead of $t$ to emphasize this is typically not a time dimension.
			Moreover, let $\widetilde{\mathcal{I}}$ denote the collection of random indexes that can depend on $(\phi_{i})_{i \in \mathcal{I}}$.
			The collection $\widetilde{\mathcal{I}}$ replaces the collection of stopping times $\mathcal{T}$ in the sequential setup.
			
			The family of test functions is then said to control the familywise error rate if
			\begin{align*}
				\sup_{\alpha} \sup_{\widetilde{\iota} \in \widetilde{\mathcal{I}}} \mathbb{E}\, \phi_{\widetilde{\iota}}(\alpha)/\alpha 
					\equiv \sup_{\alpha}\mathbb{E} \sup_{i \in \mathcal{I}}	\phi_i(\alpha)/\alpha 
						\leq 1.
			\end{align*}
			
			The familywise error rate easily generalized to the post-hoc familywise error rate, by moving the supremum over $\alpha$ inside the expectation.
			\begin{dfn}
				A collection of test functions $(\phi_i)_{i \in \mathcal{I}}$ is said to control the post-hoc familywise error rate if
				\begin{align*}
					\sup_{\widetilde{\iota} \in \widetilde{\mathcal{I}}} \mathbb{E}\sup_{\alpha} 
					\phi_{\widetilde{\iota}}(\alpha)/\alpha 
						\equiv \mathbb{E} \sup_{\alpha} \sup_{i \in \mathcal{I}} 
					\phi_i(\alpha)/\alpha 
						\leq 1.
				\end{align*}
			\end{dfn}
			
			In order to write this in terms of $p$-values, let us consider the maximum test function $\overline{\phi}$, defined as $\overline{\phi}(\alpha) = \sup_i \phi_i(\alpha)$ and assume that it is indeed a test function, which it is under the mild condition that it still attains both 0 and 1.
			This is satisfied, for example, if we consider a finite number of tests.
			We can then equivalently say that the family of test functions $(\phi_{i})_{i \in \mathcal{I}}$ controls the (post-hoc) familywise error rate if $\overline{\phi}$ is (post-hoc) valid.
			Moreover, as $\overline{\phi}$ is assumed to be a test function, it has a $p$-value $\overline{p}$.
			We can then formulate the familywise error rate in terms of this $p$-value, as captured in Theorem \ref{thm:fwer}.
			
			\begin{thm}\label{thm:fwer}
				The family of test functions $(\phi_{i})_{i \in \mathcal{I}}$ controls the (post-hoc) familywise error rate if $\overline{p}$ is a (post-hoc) valid $p$-value.
			\end{thm}
			
		\subsection{Post-hoc false discovery rate}
			We can weaken the familywise error rate by replacing the supremum $\sup_i$ in Section \ref{sec:fwer} by an expectation $\mathbb{E}_{\widetilde{\iota}}$ over a random index $\widetilde{\iota}$ on $\mathcal{I}$.
			This yields an expected multiple testing error bound:
			\begin{align*}
				\sup_{\alpha}  \mathbb{E}\, \mathbb{E}_{\widetilde{\iota}} \phi_{\widetilde{\iota}}(\alpha)/\alpha 
						\leq 1.
			\end{align*}
			If $\mathcal{I}$ is non-empty and finite, and $\widetilde{\iota}$ is uniform on $\mathcal{I}$, then $\mathbb{E}_{\widetilde{\iota}} \phi_i(\alpha) = \sum_{i=1}^{|\mathcal{I}|}\phi_i(\alpha) / |\mathcal{I}|$ is also called the \emph{false discovery proportion} of the tests $(\phi_i(\alpha))_{i \in \mathcal{I}}$.
			
			A post-hoc variant is obtained by moving $\sup_{\alpha}$ inside the outermost expectation.
			Moreover, $\widetilde{\phi} := \mathbb{E}_{\widetilde{\iota}}\phi_{\widetilde{\iota}}$ is a randomized test function, as it takes value in $[0, 1]$ (see Appendix \ref{sec:randomized}).
			Such a randomized test function has an associated $p$-function $\widetilde{p}$, so that post-hoc expected multiple testing errors can be equivalently formulated as $\widetilde{p}$ being a post-hoc $p$-function.
	
			\begin{thm}
				Let $\widetilde{\iota}$ be some data-independent random index on $\mathcal{I}$.
				Then, $(\phi_{i})_{i \in \mathcal{I}}$ controls the (post-hoc) expected multiple testing error with respect to $\widetilde{\iota}$ if and only if $\widetilde{p}$ is a (post-hoc) valid $p$-function.
			\end{thm}
			
			Clearly, the expectation with respect to $\widetilde{\iota}$ can be replaced by other functions of $(\phi_i)_{i\in\mathcal{I}}$ to obtain other types of (post-hoc) multiple testing errors.
			Related definitions appear in \citet{wang2022false}, \citet{katsevich2020simultaneous} and \citet{wang2022false}.

	\section{Post-hoc max size distortion}\label{apn:post-hoc_max}
		In this section, we return to the control of the maximum size distortion, as in Section \ref{sec:max_expected}.
		In Theorem \ref{thm:impossible}, we show that it is unfortunately impossible to do meaningful post-hoc $\alpha$ hypothesis testing without having any size distortion at any realization of the data-dependent level $\widetilde{\alpha}$.
		This result shows that only uninteresting test functions and $p$-values would satisfy such a guarantee.
		
		The underlying issue is that the maximum size distortion only considers the worst case value that $\widetilde{\alpha}$ may take, even if it is extremely unlikely.
		This makes control of the maximum size distortion overly conservative.
						
		\begin{thm}\label{thm:impossible}
			The following statements are equivalent:
			\begin{itemize}
				\item A test function $\phi$ controls the maximum size distortion for every data-dependent level $\widetilde{\alpha}$,
				\item $\phi(\alpha) = 0$ for all $\alpha \in (0, 1)$,
				\item its $p$-value $p$ satisfies $p \geq 1$.
			\end{itemize}	
		\end{thm}			
		\begin{proof}[Proof of Theorem \ref{thm:impossible}]
			The equivalence of the second and third statement follows from the definition of the $p$-value.
			
			Next, we show that the first statement implies the third.
			Suppose that $\phi$ controls the maximum size distortion for every data-dependent level $\widetilde{\alpha}$.
			A potential data-dependent level would be the smallest level at which it rejects: $\widetilde{\alpha} = p$. 
			This choice yields 
			\begin{align*}
				\sup_{a\, \in\, \textnormal{support}(p)}& \mathbb{P}(\phi(p) = 1 \mid p = a)/a \\
				&= \sup_{a\, \in\, \textnormal{support}(p)} \mathbb{E}( 1 \mid p = a)/a \\
				&= \sup_{a\, \in\, \textnormal{support}(p)} 1 / a = 1 / \inf\{a : a\in \textnormal{support}(p)\} \\
				&\leq 1
			\end{align*}
			This implies $p \geq 1$. Since $p$ is the smallest data-dependent level at which $\phi$ rejects, it cannot reject for $\alpha < 1$.
			
			Finally, suppose that $\phi(\alpha) = 0$ for all $\alpha \in (0, 1)$.
			This implies $\mathbb{P}(\phi(\widetilde{\alpha}) = 1 \mid \widetilde{\alpha} = a) = 0 \leq a$ for all $a \in (0, 1)$, for any data-dependent level $\widetilde{\alpha}$.
			Moreover, $\mathbb{P}(\phi(\widetilde{\alpha}) = 1 \mid \widetilde{\alpha} = a) \leq a$ for all $a \geq 1$, since $\mathbb{P}(\phi(\widetilde{\alpha}) = 1 \mid \widetilde{\alpha} = a) \leq 1$ as it is a probability.
			Hence, $\mathbb{P}(\phi(\widetilde{\alpha}) = 1 \mid \widetilde{\alpha} = a) \leq a$ for all $a > 0$.
		\end{proof}
			
	\section{Randomized post-hoc validity}\label{sec:randomized}
		(Potentially) randomized tests $\widetilde{\tau}$ take value in $[0, 1]$, rather than $\{0, 1\}$.
		The value of a randomized test can be interpreted as the conditional rejection probability, given the data.
		We can convert a randomized test into a binary decision by rejecting if $U \leq \widetilde{\tau}$, where $U \sim \text{Unif}(0, 1]$ independently.
		This is `reversible' by integrating out $U$: $\widetilde{\tau} = \mathbb{E}_U\mathbb{I}\{U \leq \widetilde{\tau}\}$.
		
		In some of the results in this section, it is relevant to distinguish between simple hypotheses which only contain a single distribution, and composite hypotheses which may contain multiple.
		For this reason, we explicitly derive our results under composite hypotheses: $\mathbb{E}^{H_0} := \sup_{\mathbb{P} \in H_0} \mathbb{E}^{\mathbb{P}}$.
		
		We say $\widetilde{\phi} : \mathbb{R}_+ \to [0, 1]$ is a randomized test function if $\widetilde{\phi}(\alpha)$ is cadlag, not constant and non-decreasing in $\alpha$.
		As with non-randomized tests, such a function is commonly said to be valid if $\sup_\alpha \mathbb{E}^{H_0} \widetilde{\phi}(\alpha)/\alpha  
				\leq 1$.
		We extend post-hoc validity to randomized test functions, by saying that such a function is post-hoc if $\mathbb{E}^{H_0} \sup_{\alpha > 0} \widetilde{\phi}(\alpha)/\alpha \leq 1$.
				
		\begin{rmk}[Order of randomization and $\alpha$ selection]
			We assume throughout this section that we \emph{first} select the level $\alpha$ and only then apply external randomization.
			In some sense, this is not truly \emph{post-hoc} selection of $\alpha$, since the selection is followed by something else.
			However, it is post-hoc in the sense that the selection is made with full knowledge of the data.
			
			We can, of course, also study randomized testing in a way that the randomization takes place before the final selection of $\alpha$, so that $\alpha$ can also be based on the external randomization.
			In fact, this simply reduces to the setting discussed in the main text, by viewing the external randomization as part of the data.
		\end{rmk}

		\subsection{$p$-functions}
			It is necessary to go beyond $p$-values if we are interested in randomized testing.
			The issue is that unlike non-randomized test functions, a randomized test function \emph{cannot} be losslessly converted into a $p$-value: the smallest value of $\alpha$ for which they hit $1$.
			This is because a randomized test function is not an indicator function in $\alpha$.
			This is illustrated in the first plot in Figure \ref{fig:functions}.
			
			For randomized testing we generalize $p$-values to $p$-functions, which we illustrate in the second plot in Figure \ref{fig:functions}.
			\begin{dfn}[$p$-function]
				For a (randomized) test function $\widetilde{\phi}$, we say that $\widetilde{p}$ is a $p$-function if $\widetilde{p}(u) = \inf\{\alpha : \widetilde{\phi}(\alpha) \geq u\}$.
				A $p$-function is non-randomized if $\widetilde{p}(u) = \widetilde{p}(1)$ for all $u \in (0, 1]$.
				If not, it is randomized.
			\end{dfn}
			The value $\widetilde{p}(u)$ can be interpreted as the smallest level at which the test would have rejected with probability $u$, had we chosen the level $\widetilde{p}(u)$.
			This means that $p(1)$ is a $p$-value, which coincides with the definition of a $p$-value for a randomized test that is used by \citet{lehmann2022testing}.
			A more technical interpretation is that a test function $\widetilde{\phi}$ is a CDF, and its $p$-function $\widetilde{p}$ is the corresponding quantile function.
			
			To convert a $p$-function into a binary decision, we can reject if $p(U) \leq \alpha$, where $U \sim \text{Unif}(0, 1]$ independently.
			Hence, $p(U)$ can be interpreted as an explicitly randomized $p$-value.

			A $p$-function of a non-randomized test is non-randomized, and therefore coincides with a $p$-value.
			This is shown in Proposition \ref{prp:p-family-value}.
			\begin{prp}\label{prp:p-family-value}
				If $\widetilde{\phi}$ is a non-randomized test function, then $\widetilde{p}(u) = \inf\{\alpha : \widetilde{\phi}_\alpha = 1\}= \widetilde{p}(1)$ for all $u$.
			\end{prp}
			\begin{proof}
				If $\widetilde{\phi}(\alpha) \in \{0, 1\}$ for all $\alpha$, then for $u > 0$ we have: $\widetilde{\phi}(\alpha) \geq u$ if and only if $\widetilde{\phi}(\alpha) = 1$. Substituting this into the definition of a $p$-function yields the result.
			\end{proof}
			
			A $p$-function can be translated back into a test function.
			In particular, first observe that $\widetilde{p}$ is callal (the reverse of cadlag) and non-decreasing, because $\widetilde{\phi}$ is cadlag and non-decreasing.
			In fact, if viewed as functions in $\alpha$ and $u$, these functions form a Galois connection: $\widetilde{\phi}(\alpha) \geq u$ if and only if $\widetilde{p}(u) \leq \alpha$.
			This shows that we can convert a $p$-function back into a test function through the map $\widetilde{p} \mapsto \sup\{u : \widetilde{p}(u) \leq \alpha\} = \sup_{u \in (0, 1]} u\mathbb{I}\{\widetilde{p}(u) \leq \alpha\}$.
			Moreover, if $\widetilde{\phi}$ is continuous and strictly increasing then $\widetilde{\phi}$ and its $p$-function are inverses of each other: $\widetilde{\phi}(\widetilde{p}(u)) = u$ and $\widetilde{p}(\widetilde{\phi}(\alpha)) = \alpha$.
			
			\begin{figure}
				\centering
				\begin{tikzpicture}[scale=0.55]
				    % Draw axes
				    \draw [<->,thick] (0,5) node (yaxis) [above] {$\phi(\alpha)$}
				        |- (5,0) node (xaxis) [right] {$\alpha$};
				
				    % Draw cadlag function
				    \draw[black, thick] (0,0) -- (1,0); % continuous segment
				    \draw[black, thick] (1,2) -- (2,2); % jump at x=1
				    \draw[black, thick] (2,2) -- (4,4); % continuous segment
				    \draw[black, thick] (4,4) -- (4.5,4); % jump at x=4
				    
				    % Draw dotted lines at jumps
					\draw[black, dotted] (1,0) -- (1,2); % continuous segment
				
				    % Draw jumps with dots
				    \fill[white] (1,0) circle (2pt);
				    \fill[black] (1,2) circle (2pt);
				    \draw[black] (1,0) circle (2pt);
				    
				    \draw[black] (4,0) -- (4,-0.2); % Tick mark
    				\node[below] at (4,-0.2) {$p(1)$}; % Label for the tick mark
    				\draw[black, dashed] (4,0) -- (4,4); % continuous segment 	
    				
    				\draw[black] (0,4) -- (-0.2,4); % Tick mark
    				\node[below] at (-0.5, 4.5) {$1$}; % Label for the tick mark		
				\end{tikzpicture}
				\begin{tikzpicture}[scale=0.55]
				    % Draw axes
				    \draw [<->,thick] (0,5) node (yaxis) [above] {$p(u)$}
				        |- (5,0) node (xaxis) [right] {$u$};
				
				    % Draw cadlag function
				    \draw[black, thick] (0,1) -- (2,1); % continuous segment
				    %\draw[black, thick] (2,1) -- (2,2); % jump at x=1
				    \draw[black, thick] (2,2) -- (4,4); % continuous segment
					
					 % Draw dotted lines at jumps
					\draw[black, dotted] (2,1) -- (2,2); % continuous segment
					
				    % Draw jumps with dots
				    \fill[white] (0,1) circle (2pt);
				    \draw[black] (0,1) circle (2pt);
				    \fill[black] (2,1) circle (2pt);
				    \fill[white] (2,2) circle (2pt);
				    \draw[black] (2,2) circle (2pt);
				    %\fill[black] (4,4) circle (2pt);
				    
    				\draw[black] (4,0) -- (4, -0.2); % Tick mark
    				\node[below] at (4, -0.2) {$1$}; % Label for the tick mark		
				\end{tikzpicture}
				\caption{Illustration of (realized) test family (left) and its associated $p$-family (right). We can see the relationship between test families and $p$-families by swapping the horizontal and vertical axes.}
				\label{fig:functions}
			\end{figure}
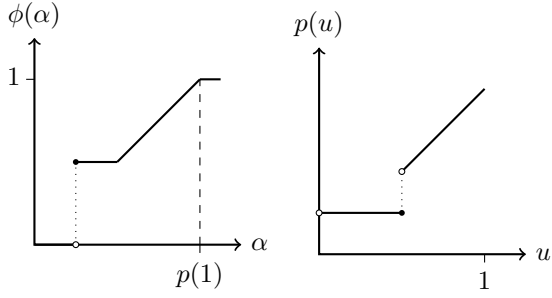
			
		\subsection{Post-hoc $p$-functions}
			We say that a $p$-function is post-hoc if its underlying test family is post-hoc.
			Alternatively, a post-hoc $p$-function can also be defined in a standalone manner, as shown in Theorem \ref{thm:ph-valid-rand-p}.
			
			Another consequence of Theorem \ref{thm:ph-valid-rand-p} is that there is no reciprocal duality between $e$-values and post-hoc $p$-functions: for a post-hoc $p$-function $\widetilde{p}$, $1/\widetilde{p}$ is not necessarily an $e$-value.
			Redefining an $e$-value as such a reciprocal, as suggested in Section \ref{sec:redefine_e-values}, would recover this duality and make $e$-values meaningful objects in the randomized testing setting.
			
			\begin{thm}\label{thm:ph-valid-rand-p}
				$\widetilde{p}$ is a post-hoc $p$-function if and only if $\mathbb{E}^{H_0} \sup_{u \in (0, 1]} u / \widetilde{p}(u)
					\leq 1$.
			\end{thm}
			\begin{proof}
				We have
				\begin{align*}
					\sup_{\alpha > 0} \widetilde{\phi}(\alpha)/\alpha
						&= \sup_{\alpha > 0} \sup_{u \in (0, 1]} u\mathbb{I}\{\widetilde{p}(u) \leq \alpha\}/\alpha \\
						&=\sup_{u \in (0, 1]} \sup_{\alpha > 0}u\mathbb{I}\{\widetilde{p}(u) \leq \alpha\}/\alpha \\
						&= \sup_{u \in (0, 1]} u / \widetilde{p}(u).
				\end{align*}
			\end{proof}
			
			Corollary \ref{cor:p_1_ph} shows post-hoc $p$-functions give rise to post-hoc $p$-values.
			\begin{cor}\label{cor:p_1_ph}
				If $\widetilde{p}$ is a post-hoc $p$-function then $\widetilde{p}(1)$ is a post-hoc $p$-value.
			\end{cor}
			\begin{proof}
				This follows from Theorem \ref{thm:ph-valid-rand-p} as $\mathbb{E}^{H_0} 1 / \widetilde{p}(1)\leq  \mathbb{E}^{H_0} \sup_{u \in (0, 1]} u / \widetilde{p}(u)$.
			\end{proof}
		
		\subsection{Randomizing post-hoc $p$-values}\label{sec:inadmissible}
			In Theorem \ref{thm:markov_equality}, we show that any post-hoc $p$-value characterizes and is characterized by a special post-hoc $p$-function.
						
			\begin{thm}\label{thm:markov_equality}
				$p$ is a post-hoc $p$-value if and only if $\widetilde{p}$ defined as $\widetilde{p}(u) = u p$ is a post-hoc $p$-function.
			\end{thm}
			\begin{proof}
				We have $\mathbb{E}^{H_0}\sup_{u \in (0, 1]} u / (u p) = \mathbb{E}^{H_0} p^{-1}$.
			\end{proof}
			
			An implication of this result is that any non-randomized post-hoc $p$-function can be trivially improved by a randomized post-hoc $p$-function.
			This is easiest to see when explicitly randomizing.
			In particular let $U \sim \text{Unif}(0, 1]$ independently of a post-hoc $p$-value $p$, then $U p$ is almost surely smaller than $p$.
			Hence the explicitly randomized $p$-value $U p$ is almost surely smaller than $p$.
			Equivalently, for a post-hoc test function $\widetilde{\phi}$, $\phi(U p)$ is valid. 
			
			The same result can also be expressed in terms of an $e$-value and post-hoc test function, as captured in Corollary \ref{cor:inadmissible_e_value}.
			This result improves a recent result by \citet{ramdas2023randomized} who show that $\alpha \mapsto \alpha e \wedge 1$ is a valid test function.
			
			\begin{cor}\label{cor:inadmissible_e_value}
				$e$ is an $e$-value if and only if $\alpha \mapsto \alpha e \wedge 1$ is a post-hoc test function.
			\end{cor}
			\begin{proof}
				We have $\sup_{\mathbb{P} \in H_0} \mathbb{E}^{\mathbb{P}} \sup_{\alpha > 0} (\alpha e \wedge 1) / \alpha  = \sup_{\mathbb{P} \in H_0} \mathbb{E}^{\mathbb{P}} \sup_{\alpha > 0} (e \wedge 1/\alpha) = \sup_{\mathbb{P} \in H_0} \mathbb{E}^{\mathbb{P}} e$.	
			\end{proof}
		
		\subsection{Merging post-hoc $p$-functions}\label{sec:merging}
			A weighted harmonic mean of post-hoc $p$-functions is post-hoc valid.
			This property extends to (data-independent) harmonic mixtures of possibly infinitely many $p$-functions.

			\begin{thm}\label{thm:harm_mean_phrp}
				A weighted harmonic mean of post-hoc $p$-functions is a post-hoc $p$-function.
			\end{thm}
			\begin{proof}
				Suppose we have $n$ $p$-functions $\widetilde{p}_i$, $i = 1, \dots, n$, and weights $w_i \geq 0$, $\sum_{i=1}^n w_i = 1$.
				Then,
				\begin{align*}
					\mathbb{E}^{H_0} \sup_{u \in (0, 1]} u \sum_{i=1}^n w_i (\widetilde{p}_i(u))^{-1}
						&\leq \mathbb{E}^{H_0} \sum_{i=1}^n w_i \sup_{u \in (0, 1]} u (\widetilde{p}_i(u))^{-1} \\\
						&\leq \sum_{i=1}^n w_i \mathbb{E}^{H_0} \sup_{u \in (0, 1]} u (\widetilde{p}_i(u))^{-1}.
				\end{align*}	
				Here, the final term is bounded by $\sum_{i=1}^n w_i = 1$ as each $\widetilde{p}$ is post-hoc.
				The same reasoning extends from weighted averages to mixtures, replacing the weighted sum by an expectation.
			\end{proof}

			The product of an arbitrary collection of independent post-hoc $p$-functions is not necessarily post-hoc valid.
			However, if they are individually all `not too randomized', then we do have that the product is post-hoc valid.
			We can express this in terms of a condition on the shape of the individual $p$-functions.
			We describe this in Theorem \ref{thm:prod_phrp}, which generalizes the result that the product of post-hoc $p$-values is post-hoc valid.
			
			The $p$-value result is recovered by setting $\widetilde{p}_i(u) = \widetilde{p}_i(1)$ for all $u \in (0, 1]$, and $i = 1, \dots, n$.
			To see that Theorem \ref{thm:prod_phrp} improves the non-randomized result, we can for example choose $p_i(u) = u^{1/n}p_i(1)$, for all $i = 1, \dots, n$.
			Since, $u \leq 1$, this dominates choosing $p_i(u) = p_i(1)$ for all $u \in (0, 1]$.
			
			\begin{thm}\label{thm:prod_phrp}
				Suppose we have $n$ independent post-hoc $p$-functions $\widetilde{p}_i$, $i = 1, \dots, n$.
				Suppose that they jointly satisfy the property $\prod_i^n \widetilde{p}_i(1)/\widetilde{p}_i(u) \leq 1/u$ for all $u \in (0, 1]$.
				Then, their product $\prod_{i=1}^n \widetilde{p}_i$ is also a post-hoc $p$-function.
			\end{thm}
			\begin{proof}
				We have
				\begin{align*}
					\sup_{\mathbb{P} \in H_0} & \mathbb{E}^{\mathbb{P}} \sup_{u \in (0, 1]} u \prod_{i=1}^n (p_i(u))^{-1}
						\leq \sup_{\mathbb{P} \in H_0} \mathbb{E}^{\mathbb{P}} \prod_{i=1}^n (p_i(1))^{-1} \\
						&= \sup_{\mathbb{P} \in H_0} \prod_{i=1}^n  \mathbb{E}^{\mathbb{P}} (p_i(1))^{-1}
						\leq  \prod_{i=1}^n  \sup_{\mathbb{P} \in H_0}  \mathbb{E}^{\mathbb{P}} (p_i(1))^{-1}
						\leq 1,
				\end{align*}	
				where the final inequality follows from Corollary \ref{cor:p_1_ph}, the second equality from independence, and the first inequality from the assumption.
			\end{proof}
			
			The condition $\prod_i^n \widetilde{p}_i(1) / \widetilde{p}_i(u) \leq 1/u$ limits how `randomized' the $p$-functions can be for large $n$.
			To see this, suppose that all $\widetilde{p}_i$ are of the same shape, then it imposes $\widetilde{p}_i(1) / \widetilde{p}_i(u) \leq u^{1/n}$.
			Moreover, $\widetilde{p}_i(1)/\widetilde{p}_i(u) \geq 1$ as $\widetilde{p}_i(u)$ is non-decreasing in $u$ by construction.
			As a consequence, we have $1 \leq \widetilde{p}_i(1) / \widetilde{p}_i(u) \leq u^{1/n}$.
			As $u^{1/n} \approx 1$ for large $n$, we have that $\widetilde{p}_i(u) \approx \widetilde{p}_i(1)$ for all $u \in (0, 1]$ if $n$ is large.
			That is, for large $n$, this condition is essentially only satisfied if the $p$-functions are non-randomized.
			
			We formalize the importance of non-randomization for product-merging in Theorem \ref{thm:prod_merge}.
			This result states that essentially only non-randomized post-hoc $p$-functions can be arbitrarily multiplied together.
			
			Here, we say that a $p$-family is `properly' randomized (with respect to $H_0$) if $\mathbb{E}^\mathbb{P} (p(1))^{-1}$ is bounded away from $\mathbb{E}^{\mathbb{P}}(p(u^*))^{-1}$ uniformly in $\mathbb{P} \in H_0$, for some $u^* \in (0, 1]$.
			
			\begin{thm}\label{thm:prod_merge}
				Suppose we have at most countably infinitely many i.i.d. copies of a randomized $p$-function $\widetilde{p}$ with $\mathbb{E}^{H_0}(\widetilde{p}(1))^{-1} = 1$.
				If every product of these copies is post-hoc valid, then the $p$-functions are not properly randomized. 
				
				If we additionally assume that $H_0$ is finite, then the $p$-functions are non-randomized.
			\end{thm}
			\begin{proof}[Proof of Theorem \ref{thm:prod_merge}]
				The strategy is to assume the $p$-function is randomized, and then show that this is in contradiction with the assumption that they are post-hoc valid.
				We only assume proper randomization at the very end, to handle the supremum over a possibly infinite $H_0$.
				
				As the $p$-function $\widetilde{p}$ is randomized, there exists some $u^*$ such that $\widetilde{p}(1) > \widetilde{p}(u^*)$.
				Let $B$ index the collection $(\widetilde{p}_i)_{i \in B}$ of i.i.d. copies of the $p$-function.
				
				First, by the post-hoc validity of the product and the by the fact that $u^*$ is not necessarily the optimizer, we have
				\begin{align*}
					1 
						&\geq \sup_{I \subseteq B} \sup_{\mathbb{P} \in H_0} \mathbb{E}^{\mathbb{P}} \sup_u u \prod_{i \in I} (\widetilde{p}_i(u))^{-1} \\
						&\geq \sup_{I \subseteq B} \sup_{\mathbb{P} \in H_0} \mathbb{E}^{\mathbb{P}} u^* \prod_{i \in I} (\widetilde{p}_i(u^*))^{-1}.
				\end{align*}
				Now, rewriting the final term yields
				\begin{align*}
					\sup_{I \subseteq B} \sup_{\mathbb{P} \in H_0} &\mathbb{E}^{\mathbb{P}} u^* \prod_{i \in I} (\widetilde{p}_i(u^*))^{-1} \\
						&= \sup_{I \subseteq B} \sup_{\mathbb{P} \in H_0} \mathbb{E}^{\mathbb{P}} \prod_{i \in I}  (u^*)^{1/|I|} (\widetilde{p}_i(u^*))^{-1}.
				\end{align*}
				Next, by the i.i.d assumption we can further rewrite it as
				\begin{align*}
					\sup_{I \subseteq B} \sup_{\mathbb{P} \in H_0} &\mathbb{E}^{\mathbb{P}} \prod_{i \in I}  (u^*)^{1/|I|} (\widetilde{p}_i(u^*))^{-1} \\
						&=\sup_{I \subseteq B} \sup_{\mathbb{P} \in H_0} \prod_{i \in I} \mathbb{E}^{\mathbb{P}} (u^*)^{1/|I|} (\widetilde{p}_i(u^*))^{-1} \\
						&= \sup_{1 \leq |I| \leq \infty} \left[\sup_{\mathbb{P} \in H_0} \mathbb{E}^{\mathbb{P}} (u^*)^{1/|I|} (\widetilde{p}(u^*))^{-1}\right]^{|I|}.
				\end{align*}
				Restricting ourselves to infinite sets, we have
				\begin{align*}
					1 &\geq \sup_{1 \leq |I| \leq \infty} \left[\sup_{\mathbb{P} \in H_0} \mathbb{E}^{\mathbb{P}} (u^*)^{1/|I|} (\widetilde{p}(u^*))^{-1}\right]^{|I|} \\
						&\geq \sup_{|I| = \infty} \left[ \sup_{\mathbb{P} \in H_0}\mathbb{E}^{\mathbb{P}} (u^*)^{1/|I|} (\widetilde{p}(u^*))^{-1}\right]^{|I|} \\
						&= \sup_{|I| = \infty} \left[ \sup_{\mathbb{P} \in H_0}\mathbb{E}^{\mathbb{P}} (\widetilde{p}(u^*))^{-1}\right]^{|I|} \\
						&= \left[\sup_{\mathbb{P} \in H_0} \mathbb{E}^{\mathbb{P}} (\widetilde{p}(u^*))^{-1}\right]^{\infty},
				\end{align*}
				since $u^* \in (0, 1]$.
				
				As $\widetilde{p}(1) > \widetilde{p}(u^*)$, we have $\mathbb{E}^{\mathbb{P}} (\widetilde{p}(u^*))^{-1} > \mathbb{E}^{\mathbb{P}} (\widetilde{p}(1))^{-1}$ for every $\mathbb{P} \in H_0$.
				This observation allows us to finish the case that $H_0$ is finite.
				In particular, as $H_0$ is finite, its supremum is attained.
				As a consequence if $\mathbb{E}^{\mathbb{P}} (\widetilde{p}(u^*))^{-1} > \mathbb{E}^{\mathbb{P}} (\widetilde{p}(1))^{-1}$ for every $\mathbb{P} \in H_0$, we have that
				\begin{align}\label{ineq:strict}
					\sup_{\mathbb{P} \in H_0} \mathbb{E}^{\mathbb{P}} (\widetilde{p}(u^*))^{-1} > \sup_{\mathbb{P} \in H_0} \mathbb{E}^{\mathbb{P}} (\widetilde{p}(1))^{-1} = 1.
				\end{align}
				As a consequence,
				\begin{align*}
					1 
						\geq \left[\sup_{\mathbb{P} \in H_0} \mathbb{E}^{\mathbb{P}} (\widetilde{p}(u^*))^{-1}\right]^{\infty}
						= \infty,
				\end{align*}
				which is a contradiction.
				
				For the case that $H_0$ is infinite, it is insufficient for \eqref{ineq:strict} to assume that $\widetilde{p}(u^*) > \widetilde{p}(1)$.
				However, the result is recovered as we assume $\mathbb{E}^{\mathbb{P}} \widetilde{p}(u^*)$ is bounded away from $\mathbb{E}^{\mathbb{P}} \widetilde{p}(1)$, uniformly in $\mathbb{P} \in H_0$. 
			\end{proof}
			
	\section{Proofs}
		\subsection{Proof of Proposition \ref{prp:minimal_h}}\label{proof:minimal_h}
			\begin{proof}
				The first claim is equivalent to Theorem \ref{thm:post-hoc_is_valid}.
			
				For the second claim, fix $h < 1$ and $q \in (0,1)$, and let $e = M > 1$ with probability $q$ and $0$, otherwise.
				Then $\rho_{h}(e) = (qM^{h})^{1/h} = q^{1/h} M$.
				The induced level-$\alpha$ test is $\phi_e(\alpha) = \alpha^{-1}\mathbb{I}\{e \geq 1/\alpha\}$.
				Now, since $e \in \{0,M\}$ we have $\mathbb{E}[\phi_e(\alpha)] = 0$ if $\alpha < 1/M$ and $\mathbb{E}[\phi_e(\alpha)] = q/\alpha$ if $\alpha \geq 1/M$, so that $\sup_{\alpha} \mathbb{E}[\phi_e(\alpha)] = qM$.
				Choosing $M = q^{-1/h}$ gives $\rho_{h}(e) = 1$ and $\sup_{\alpha} \mathbb{E}[\phi_e(\alpha)] = qM = q^{1 - 1/h} > 1$, since $1 - 1/h < 0$.
				Hence, $e$ is $h$-valid but $\phi_e$ is not valid.
			\end{proof}
		
		\subsection{Proof of Proposition \ref{prp:harmonic_size_difference}}\label{proof:harmonic_size_difference}
			\begin{proof}
				\begin{align*}
					&\sup_{\widetilde{\alpha}}\mathbb{E}_{\widetilde{\alpha}}[\mathbb{P}(\phi(\widetilde{\alpha}) = \textnormal{reject at } \widetilde{\alpha} \mid \widetilde{\alpha}) - \widetilde{\alpha}] \\
						&= \sup_{\widetilde{\alpha}}\mathbb{E}_{\widetilde{\alpha}}[\mathbb{E}[\mathbb{I}\{\phi(\widetilde{\alpha}) = \textnormal{reject at } \widetilde{\alpha}\}  \mid \widetilde{\alpha}] - \widetilde{\alpha}] \\
						&= \sup_{\widetilde{\alpha}}\mathbb{E}_{\widetilde{\alpha}}[\mathbb{E}[\mathbb{I}\{\phi(\widetilde{\alpha}) = \textnormal{reject at } \widetilde{\alpha}\} - \widetilde{\alpha}  \mid \widetilde{\alpha}] ] \\
						&= \sup_{\widetilde{\alpha}}\mathbb{E}[\mathbb{I}\{\phi(\widetilde{\alpha}) = \textnormal{reject at } \widetilde{\alpha}\} - \widetilde{\alpha}  ] \\
						&= \mathbb{E}\left[\sup_{\alpha : \phi(\alpha) = \textnormal{reject at } \alpha} \mathbb{I}\{\phi(\alpha) = \textnormal{reject at } \alpha\} - \alpha\right] \\
						&= \mathbb{E}[1 - \inf\{\alpha : \phi(\alpha) = \textnormal{reject at } \alpha\}] 
						= \mathbb{E}[1 - p].
				\end{align*}
				Hence, $\phi$ is post-hoc valid in this sense if and only if $\mathbb{E}[p] \geq 1$.
				Equivalently, its $e$-value $e = 1/p$ is harmonic.
			\end{proof}
	
	\section{Measurability and conditioning}\label{sec:measurability_conditioning}
		\subsection{Measurability}		
			We define measurability with respect to an information structure $\mathcal{I} \subseteq 2^{\mathcal{X}}$.
			To define this, we write $x \sim_{\mathcal I} y$ if and only if for every $A \in \mathcal I$ we have $x \in A \iff y \in A$.
			For each $x \in \mathcal X$, we call
			\begin{align*}
				C_{\mathcal I}(x) := \{ y \in \mathcal X : y \sim_{\mathcal{I}} x \}
			\end{align*}
			the \emph{information cell} of $x$ induced by $\mathcal{I}$.
			The family of such information cells $\mathcal{C}(\mathcal{I})$ partitions $\mathcal{X}$.
			
			\begin{dfn}[Measurability] 
				We say $\varepsilon$ is $\mathcal{I}$-measurable if it is constant on each information cell:  $x, y \in C \implies \varepsilon(x) = \varepsilon(y)$, for every $C \in \mathcal{C}(\mathcal{I})$. 
			\end{dfn}
			
			Under this notion of measurability, we need not worry about the measurability of a supremum of measurable evidence variables.
			
			\begin{lem}
				The supremum $\sup_{i} \varepsilon_i$ of a collection of measurable evidence variables is measurable.
			\end{lem}
			\begin{proof}
				By $\mathcal{I}$-measurability, $\varepsilon_i$ is constant on every cell $C \in \mathcal{C}(\mathcal{I})$, and so the supremum is as well.	
			\end{proof}		
							
			\begin{rmk}
				The notion of $\mathcal{I}$-measurability here can be viewed as classical measurability with respect to the $\sigma$-algebra $\Sigma_{\mathcal{I}}$ of unions of information cells $C \in \mathcal{C}(\mathcal{I})$, which is the atomic $\sigma$-algebra with the information cells as atoms.
				If $\varepsilon$ is $\mathcal{I}$-measurable in our sense, then for any $\sigma$-algebra $\Sigma_\mathcal{D}$ on $\mathcal{D}$ the map $\varepsilon : (\mathcal{X}, \Sigma_{\mathcal{I}}) \to (\mathcal{D}, \Sigma_{\mathcal{D}})$ is classically measurable.
				Conversely, if $\Sigma_{\mathcal{D}}$ separates points, then classical measurability of $\varepsilon$ implies it is constant on information cells, and hence $\mathcal{I}$-measurable in our sense.
			\end{rmk}
			
		\subsection{Conditioning}
			Using this notion of measurability, we can capture the axiom of replacement of \citet{kolmogorov1930notion}.
			To our surprise, this axiom may be viewed as defining a concept of conditional expectation on total orders, which in itself may be seen as defining a notion of compatibility across information structures.
		
			Given two $\mathcal{I}$-measurable evidence variables $\varepsilon, \eta$ and $A \in \mathcal{I}$, we define their $A$-mixture as
			\begin{align*}
				(\varepsilon \oplus_A \eta)(x)
					:=
					\begin{cases}
						\varepsilon(x), & x \in A, \\
						\eta(x), & x \not\in A.	
					\end{cases}
			\end{align*}
			The resulting mixture $\varepsilon \oplus_A \eta$ of evidence variables is again $\mathcal{I}$-measurable for $A \in \mathcal{I}$.
			
			Let $\mathcal{S} \subseteq \mathcal{I}$ be a sub-information-structure of $\mathcal{I}$.
			A \emph{certainty equivalent kernel} for $\mathcal{S}$ is a mapping $(\varepsilon, x) \mapsto \rho(\varepsilon \mid \mathcal S)(x)$ such that, for every $\mathcal{I}$-measurable evidence variable $\varepsilon$, the map $x \mapsto \rho(\varepsilon \mid \mathcal S)(x)$ is $\mathcal{S}$-measurable, and for every $x \in \mathcal X$, the map $\varepsilon \mapsto \rho(\varepsilon \mid \mathcal S)(x)$,
			restricted to $\mathcal S$-measurable evidence variables, is an idempotent and monotone certainty-equivalent operator.
			
			We are now ready to introduce an order-theoretic version of the replacement axiom of \citet{kolmogorov1930notion}.
			The term `replacement' comes from the fact that we replace $\varepsilon$ on a set $A$ by an aggregation $\rho(\varepsilon \mid \mathcal{S})$.
			
			\begin{itemize}
				\item[]	\textbf{Replacement} \\
					There exists a unique certainty equivalent kernel $\rho(\cdot \mid \mathcal{S})$ for $\mathcal{S}$ so that for every $\mathcal{I}$-measurable evidence variable $\varepsilon$,
					\begin{align*}
						\rho(\rho(\varepsilon \mid \mathcal{S}) \oplus_A \varepsilon)
							= \rho(\varepsilon), \text{ for every } A \in \mathcal{S}.
					\end{align*}
			\end{itemize}
			
			We leave it to future work to explore sequential testing on total orders under information-structure filtrations.

\end{document}